\documentclass[12pt,a4paper]{amsart}

\usepackage{amsmath}
\usepackage{amsthm}
\usepackage{amssymb}
\usepackage{amsfonts}
\usepackage{latexsym}
\usepackage{mathtools}
\usepackage{graphics}
\usepackage{mathabx}

\usepackage[margin=2cm]{geometry}
\DeclareMathOperator{\intt}{int}

\DeclareMathOperator{\sign}{sign}

\usepackage[dvipsnames]{xcolor}

\definecolor{blue}{rgb}{0,0,0.8}
\definecolor{red}{rgb}{0.8,0,0}
\definecolor{darkgreen}{rgb}{0,0.6,0}

\newcommand\restr[2]{{  \left.\kern-\nulldelimiterspace
  #1
  \vphantom{\big|}
  \right|_{#2}
  }}

\newcommand{\oS}{\overline{S}}

\newcommand{\DS}{\partial S}

\newcommand{\PP}{{\mathcal P}}
\newcommand{\Pnn}{{\mathcal P}_{n-1}} 
\newcommand{\Pn}{{\mathcal P}_n}
\newcommand{\WW}{{\mathcal W}}
\newcommand{\YY}{{\mathcal Y}}

\newcommand{\A}{{\mathcal A}}

\newcommand{\W}{{\mathcal W}}

\newcommand{\N}{{\mathcal N}_n}
\newcommand{\OO}{{\mathcal O}_n}

\newcommand{\Ji}{{\mathcal J}_i}
\newcommand{\Jj}{{\mathcal J}_j}

\newcommand{\Kij}{{\mathcal K}_{ij}}

\newcommand{\LL}{{\mathcal L}}

\newcommand{\ai}{{\alpha}_i}

\newcommand{\lk}{{\ell}_k}

\newcommand{\RR}{{\mathbb R}}

\newcommand{\II}{{\mathbb I}}

\newcommand{\de}{{\delta}}

\newcommand{\ff}{\varphi}
\newcommand{\al}{{\alpha}}

\newcommand{\ve}{{\varepsilon}}
\newcommand{\vk}{{\varepsilon}_k}

\newcommand{\vki}{{\varepsilon}_{k,i}}

\newcommand{\vnni}{{\varepsilon}_{n+1,i}}

\newcommand{\yy}{\mathbf{y}}

\newcommand{\zz}{\mathbf{z}}

\newcommand{\xx}{\mathbf{x}}

\newtheorem{theorem}{Theorem}
\newtheorem{corollary}{Corollary}
\newtheorem{lemma}{Lemma}

\newtheorem{definition}{Definition}

\theoremstyle{definition}
\newtheorem{example}{Example}
\newtheorem{remark}{Remark}

\begin{document}

\title[Examples and counterexamples about interpolation]{On some examples and counterexamples\\  about weighted Lagrange interpolation \\ with Exponential 
and Hermite weights}

\author{Patricia Szokol}

\address{Faculty of Informatics, University of Debrecen, P.O. Box 400, H-4002 Debrecen, Hungary, HUN-REN Alfr\'ed R\'enyi Institute of Mathematics, P.O. Box 127, H-1364, Budapest, Hungary and HUN-REN-UD Equations, Functions, Curves and their Applications Research Group}

\begin{abstract}
The famous Bernstein conjecture about optimal node systems in classical polynomial Lagrange interpolation, standing unresolved for about half a century, was solved by T. Kilgore in 1978. Immediately after him, the additional conjecture of Erd\H os was also solved by de Boor and Pinkus. These breakthrough achievements were built on a fundamental auxiliary result on the nonsingularity of derivative (Jacobian) matrices of certain interval maxima in the function of the nodes.

After the above breakthrough, a considerable effort was made to extend the results to the case of at least certain Chebyshev-Haar spaces of functions. 
Here, we analyze to what extent the key nonsingularity statement remains true in the case of exponentially weighted interpolation on the half-line, or with Hermite weights on the full real line. In these settings, counterexamples demonstrate that the respective derivative matrices may as well be singular.
It remains to further study whether the Bernstein- and Erd\H{o}s characterizations remain valid.

The ``hybrid'' Chebyshev-Haar system of exponentially weighted polynomials adjoined with constant functions and the corresponding interpolation were previously studied, as well. Some hints were also given for the proof of the respective Bernstein and Erdős conjectures. We present in detail the full proof together with all the auxiliary results needed in this setting.  
\end{abstract}

\maketitle



\section{Introduction}\label{sec:Intro}

\subsection{Notation}  
Denote by $\II$ any interval, the classical case being a compact interval $[a,b]$. Here we will talk about $[0,\infty)$ and $\RR=(-\infty,\infty)$ as well. Also, let a weight $w(t)>0$ be given. In this paper we will deal with rather special weights only: our weights are $\exp(-t)$ for $[0,\infty)$ (the ``exponential weight'') and $\exp(-t^2)$ for $\RR$ (the ``Hermite weight''). In the classical, unweighted case $w(t)\equiv 1$. We will denote the space of degree at most $n$ algebraic polynomials as $\Pn$, and when these are weighted by $w(t)$, then we consider the space $\WW_n:=\{ pw~:~ p\in\Pn\}$. In particular, following \cite{K-AMH}, the weighted polynomial space with the exponential weight $w(t)=\exp(-t)$ -- where here we automatically take the underlying interval as $[0,\infty)$ -- will be denoted\footnote{The spaces $\WW_n$ with weight $w(t)=e^{-t}$ and $Y_n$ differ slightly: $\WW_n$ consists of all weighted polynomials on the whole real line $\RR$, while $Y_n$ is a subspace of continuous functions on the positive half-line that vanish at $\infty$.} as $Y_n$, and the space of weighted polynomials with the Hermite weight and considered on the full real line will be denoted as $Z_n$.

Our aim is to study the norm of the weighted Lagrange interpolation operator $\LL\colon C_w(\II)\to \WW_{n}$, where $C_w(\II)=\{f\in C(\II): \|f\|_w:=\|fw\|_{\infty}<\infty\}$ and to determine the optimal node system $\xx=(x_1,\ldots,x_n)$ which minimizes the operator norm. A node system $\xx=(x_1,\ldots,x_n)$ is an ordered array of points $a<x_1<\dots<x_n<b$, where $x_i\in \II$. All such node systems form \emph{the simplex} $S:=\{\xx ~:~ a<x_1<\dots<x_n<b\}$. When the weight $w$ is nonzero at the finite endpoints $a, b$, interpolation is also performed at the endpoints. Accordingly, we set $x_0:=a$ and $x_{n+1}:=b$, which belong to the node system $\xx$ as fixed nodes\footnote{In most cases considered in this manuscript, we apply the interpolation at $a=:x_0$ and $x_1<\dots<x_n$, but not at $b:=\infty$ and hence $\LL\colon C_w(\II)\to \WW_n$. However, in the case of Hermite weight, where $\II=\RR$, the interploation is applied at the nodes $x_1<\dots<x_n$, i.e., $\LL\colon C_w(\II)\to Z_{n-1}$.}.

There is also a special case of ``interpolation at infinity'', namely when we formally consider interpolation at $b=\infty$, even if in the classical sense this has no meaning. To give this a meaning, we consider only functions with a definite limit at infinity, and substitute for the function value this limit: $f(\infty):=\lim_{t\to +\infty} f(t)$. That special consideration will occur in the framework of the extended weighted polynomial system, when exponentially weighted degree (at most) $n$ polynomials are combined with constant functions forming the ``hybrid function system'' $\YY_n$. 

Let $h_k(t):=h_k(w,\xx,t)$ denote the \emph{fundamental polynomials} of the system corresponding to the nodes of $\xx$. (These are denoted by $y_k$ in \cite{K-AMH}). These are (weighted) polynomials satisfying $h_k(x_i)=\delta_{k,i}$ with the Kronecker delta function: $\delta_{k,i}=1$, if $k=i$, and 0 otherwise. The \emph{Lebesgue function} is then defined by 
\[
L(\xx,t):=L(w,\xx,t):=\sup_{\|f\|_w<1}|\LL(f)(t)|=\sum\limits_{k=0,1}^{n,n+1}|h_k(t)|.
\]
Here, the sum starts from $0$ or $1$, respectively ends at $n$ or $n+1$,  corresponding to the case if we interpolate at the point $a$, respectively $b$, or not.

Due to the sign changes of the $h_k$, in any of the intervals\footnote{As was said above, $x_0=a$ and $x_{n+1}=b$, if we interpolate at the endpoints. Otherwise, let $x_0:=-\infty$ and $x_{n+1}:=\infty$.} $I_i:=[x_{i-1},x_i]$, $i=1,\dots, n+1$, (with $I_0:=(-\infty,a]$, if we interpolate at $a$ and hence $x_0=a$) the Lebesgue function consists of sums of the form $\sum_{k=0,1}^{n,n+1} \pm h_k(t)$, with the choice of signs $\pm$ depending on the interval $I_i$, but otherwise remaining constant within one such interval. As it is well-known, the norm of the operator $\LL\colon C_w(\II)\to Y_n, \hspace{2mm} (Z_n, \YY_n)$ coincides with the supremum norm of $L(\xx,\cdot)$ on $\II$. Therefore, we consider the interval maxima $m_i(\xx):=\max_{t \in I_i} L(\xx,t)$ (for all occuring $i$).  The \emph{optimal nodes} of interpolation are those node systems that minimize the operator norm, i.e., that minimize
$$\|\LL\|_w=\|L(\xx,\cdot)\|_{\infty}=\max\limits_i m_i(\xx).$$

\subsection{Basic properties of interpolation}

Let us summarize the basic properties of the interpolation processes this paper is about.

\subsubsection{Equioscillation property}\label{Bernstein} 

If all $m_i$ values are equal, we say that the Lebesgue function of the Lagrange interpolation equioscillates. Bernstein \cite{Bernstein} conjectured that in the classical, unweighted case, the optimal (minimax) node system equioscillates, i.e., every optimal node system is equioscillating. Later, the unicity statement is often considered part of the Bernstein Conjecture–there exists exactly one equioscillating node system–although Bernstein himself did not formulate this part of the Conjecture.

It seems to be more natural to separate these claims, and call the statement ``optimality entails equioscillation'' as the \emph{equioscillation property} of the interpolation. 

\begin{definition}[Equioscillation property]
We say that the Bernstein equioscillating property holds if any optimal (minimax) node system $\xx$ necessarily equioscillates.
\end{definition}

Proving the converse of the equioscillation property and unicity requires further investigation; the following properties will be key to obtaining them, whereas deriving the direct statement usually results from a perturbation argument. Already, the statement that (for the classical polynomial interpolation on a compact interval) there exists an equioscillation node system, is highly nontrivial--it was the main result of \cite{Kilgore-BullAMS}, which summarized the findings of Kilgore's PhD thesis \cite{KilgorePhD}. The Bernstein Conjecture, direct and converse, was then fully proved in \cite{Kilgore}.

\subsubsection{``Sandwich'' property}\label{Sandwich} Erd\H{o}s \cite{Erdos-1, Erdos-2} coupled the above Bernstein conjecture with the following statement. Let $\yy$ be an (arbitrary) optimal node system, i.e., one for which $\|L(\yy,\cdot)\|_{\infty}$ is minimal. According to the Bernstein conjecture, such a system is equioscillating and satisfies
\[
\|L(\yy,\cdot)\|_{\infty} = m_i(\yy)
\quad \text{for all relevant indices } i.
\]

This leads to the following property.

\begin{definition}[Sandwich property]
We say that the (Erd\H os or) sandwich property holds if for the equioscillating node system $\yy$ and for any node system $\xx \in S$ with $\xx \neq \yy$, there exist indices $j,k$ such that
\[
m_j(\xx) < \|L(\yy,\cdot)\|_{\infty}
\quad \text{and} \quad
m_k(\xx) > \|L(\yy,\cdot)\|_{\infty}.
\]
\end{definition}

In other words, for any non-optimal node system, some of the corresponding quantities fall below, while others exceed the optimal supremum norm. This notion is referred to as the \emph{sandwich property} by Szabados and Vértesi (see p.~96 of \cite{SzV}). In the classical case, it was proved by de Boor and Pinkus \cite{CBoorPinkus}, building on the fundamental analysis of Kilgore \cite{Kilgore}.

This result has strong consequences, as, e.g., it contains the unicity statement of a minimax (optimal) system, in particular, and it also means that $\min_i m_i(\xx)$ and $\max_i m_i(\xx)=\|L(\xx,\cdot)\|_{\infty}$ provide lower and upper, respectively, bounds for the extremal value $\|L(\yy,\cdot)\|_{\infty}$. Therefore, this is a strong tool to estimate the optimal (least possible) norm of interpolation by determining these lower and upper bounds. Given that optimal node systems are hard to find and no exact description exists for general $n$, these methods of estimation are of crucial importance.

\subsubsection{Nonsingularity}\label{Nonsingularity} 

Although the following property is usually not formulated as a separate theorem, it has been a crucial step in essentially all proofs concerning optimal interpolation. 

\begin{definition}[Nonsingularity property]\label{def:nonsing}
Let the interval maxima $m_i$ be (continuously) differentiable in function of the nodes $\xx\in S$. We say that the \emph{Nonsingularity property} holds if the derivative matrices
\begin{equation}\label{Adef}
A:=\left[ \frac{\partial m_i}{\partial x_j}\right]_{i=1; j=1}^{n+1 ; n}, 
\qquad 
A_k:=\left[ \frac{\partial m_i}{\partial x_j}\right]_{i=1,\, i\ne k; j=1}^{n+1 ; n}
\end{equation}
have the property that each matrix $A_k$ is nonsingular for all $\xx \in S$ and all $1 \le k \le n+1$.
\end{definition}

The breakthrough paper of Kilgore \cite{Kilgore}, as well as subsequent works, relies on this key auxiliary result, with most of the effort devoted to establishing this property.

Once proven, this property has far-reaching consequences, following much more easily than the proof of this fact itself. See in particular the elegant and interesting analysis by Shi in \cite{Shi}, where he demonstrates that once this property is known, all other consequences can be derived via linear programming techniques.

\subsubsection{Homeomorphism}\label{Homeomorphism} Following de Boor and Pinkus \cite{CBoorPinkus}, we introduce the difference of consecutive interval maxima function
\[
\Gamma(\xx) := (\ldots, m_i(\xx) - m_{i-1}(\xx), \ldots).
\]
This defines a vector-valued mapping from $S$ into $\RR^n$, of the same dimension.

\begin{definition}[Homeomorphism property]\label{def:hom}
We say that the mapping $\Gamma$ satisfies the \emph{homeomorphism property} if it is a homeomorphism of $S$ onto $\RR^n$.
\end{definition}

In the classical case, this property was discovered and proved by de Boor and Pinkus \cite{CBoorPinkus}, who in fact considered it the main result of their paper.

Observe that once this is known, it also has important consequences. In particular, the preimage of the identically zero vector ${\bf 0}$ under $\Gamma$ is the set of all equioscillating node systems, hence if we already know that a minimax system is equioscillating, the homeomorphism property implies uniqueness of the optimal node system.

\subsubsection{Intertwining property}\label{Intertwining} As a far-reaching strengthening of Property~\ref{Sandwich}, we introduce the following notion.

\begin{definition}[Intertwining property]\label{def:int}
We say that the \emph{intertwining property} holds if for any two distinct node systems $\xx, \yy \in S$, $\xx \neq \yy$, there exist indices $i,j$ such that
\[
m_i(\xx) < m_i(\yy)
\quad \text{and} \quad
m_j(\xx) > m_j(\yy).
\]
\end{definition}

Thus, similar comparison relations hold not only between an arbitrary node system and an optimal one, but between any two distinct node systems.

The intertwining property is stronger than the Erd\H{o}s--sandwich property, and in particular implies the uniqueness of the equioscillating system. In the classical case of Lagrange interpolation by polynomials on a compact interval, this property was proved by de Boor and Pinkus \cite{CBoorPinkus}.

\subsection{Results}

The present work stems from our effort to analyze the paper \cite{K-AMH}.We considered ways to generalize the results and, therefore, sought to better understand the methods for proving the above properties and their relationship.

In the classical case, as shown by the landmark works \cite{Kilgore} and \cite{CBoorPinkus}, all the above five properties were proven. It was clear already from these works that the strongest of these properties, from which the others are easier, is the Nonsingularity Property \ref{Nonsingularity}. As the later analysis of Shi \cite{Shi} exposes, only a slight additional observation (that of \emph{properness} of the mapping $\Gamma$) is necessary for getting the homeomorphism property. Also, the equioscillation and the intertwining properties can be deduced directly from the nonsingularity property. Therefore, it is no wonder that all cases of interpolatory function systems, where the truth of the Bernstein Conjecture was confirmed, also admit all the other properties--they are always  proved through the nonsingularity property, which then brings with itself the rest.

However, here we will see that--contrary to some claims in \cite{K-AMH}--exponentially weighted polynomials on the half-line, and weighted polynomial systems with the Hermite weight $\exp(-t^2)$ on the real line, do not admit this strong nonsingularity property. As a result, for these systems, the mutatis mutandis proofs of the remaining properties (such as the Bernstein Conjecture, for example), necessarily break down. Here we will show that the validity of some of the above properties--namely, the intertwining property--also fails to hold for these systems. It remains for future investigations if the equioscillation property and the sandwich property can be saved, but if so, then new ideas will be necessary to overcome the singularity of the arising derivative matrices.

A vague idea about a point of attack could be to prove and then exploit better the ``next best thing'' to nonsingularity, that is, the homeomorphism property of the ``difference of local maxima function'' $\Gamma$, see formula \eqref{gamma}. This homeomorphism property was proved in the classical case by de Boor and Pinkus as Theorem 2 of \cite{CBoorPinkus}. The proof invokes a classical, but nontrivial, topology theorem due to Hadamard: a locally nonsingular proper mapping is a covering map, i.e., in our case, a homeomorphism of the set of node systems $S$ onto $\RR^n$. Directly inspired by the results of de Boor and Pinkus and the case of classical Lagrange interpolation, this homeomorphism property was successfully proved and applied in a number of situations where the nonsingularity property is unavailable, or at least is not known. See, e.g., \cite{Bojanov}, \cite{TLMS2018, Homeo, JMAA}, \cite{Tatiana-I, Tatiana-II}.

In \cite{K-AMH}, the Lagrange interpolation for the so-called ``hybrid system'' $\YY_n:={\rm span}\{Y_n, {\bf 1}\}$, i.e., the space of weighted polynomials with exponential weight adjoined with constant functions, is also studied. The proof of the Bernstein and Erdős conjectures for this setting (see Theorem 2 in \cite{K-AMH}) is rather sketchy, but some steps can be recovered from the hints of the author (or from the related literature, where the author does not refer to them precisely). 

Most importantly, in the case $\YY_n$, we encounter a Chebyshev–Haar system that contains the constant functions; hence, by Lemma 1 of \cite{Kilgore-Cheney}, at each point $t$ that is not a node we have $L(\xx,t)>1$. As we will see, this property will ultimately guarantee the nonsingularity of the matrices A in a nontrivial way.

Nevertheless, the argument about ``interpolation at infinity'' requires a detailed explanation. 

A warning sign is that the straightforward zero-counting method of proof, working so far, may encounter difficulties, because in $\WW_n$ it is clear that a function $f \in Y_n$ either has $n$ real roots or it has at most $n-2$, but in $\YY_n$ it is not determined if a function $f\in\YY_n$ has $n+1$ or, well possibly, $n$ real roots. (E.g., if the constant term is zero, the remaining weighted polynomial of degree $n$ can have exactly $n$ real zeros. If $n=1$, this means 1 zero, whereas $-3te^{-t}+1$ has already 2, having a negative value at $t=1$ and attaining  positive values at $0$ and $2$.)  

The structure of this paper is as follows. Section 2 is devoted to the presentation of some counterexamples that show the inaccuracies of the proof of the Bernstein and Erdős Conjectures in \cite{K-AMH} for weighted Lagrange interpolation with Exponential- and Hermite weights. In Section 3,  we consider weighted Lagrange interpolation with exponential weight for the new hybrid system setup $\YY_n$. The corresponding results were also stated by Kilgore in \cite{K-AMH}, although only with some hints of the proof. A complete verification of the individual steps, however, requires substantially more elaborate reasoning. We present the proof of the nonsingularity property in detail and, invoking the corresponding results of Shi, show that the additional basic properties of interpolation processes hold as well, including the Bernstein and Erdős Conjectures.

\section{Counterexamples about weighted Lagrange interpolation with Exponential- and Hermite weights}

\subsection{Interpolation on the half-line with exponential weight}\label{sec:exp}

Here our weight is $w(t):=\exp(-t)$, and the weighted polynomial system $Y_n=\{\exp(-t)p(t)~:~ p\in \PP_n\}$ is considered on the half-line $\II:=[0,\infty)$. If an interpolatory node system $\xx=(x_0,x_1,\ldots,x_n)$ is given, where $x_0:=0$, then we  consider the interval maxima $m_i$ ($i=1,\ldots,n+1)$, together with the location(s) $z_i(\xx)$, where the maximum $m_i(\xx):=\max_{I_i} L(\xx,\cdot)$ is attained within $I_i(\xx):=[x_{i-1},x_i]$ for $i=1,\dots,n$ and $I_{n+1}=[x_n,\infty)$.

Theorem 1 of \cite{K-AMH} claims that the Bernstein and Erd\H{o}s Conjectures hold true for this system. The proof of Theorem 1 relies on the investigation of the matrix $A$ and its submatrices $A_k$, as given in \eqref{Adef}. The latter are declared to be nonsingular, and this property plays a crucial role throughout the argument. However, there is an oversight in the proof, and in fact, the stated nonsingularity fails for many node systems.

The key overlooking is in the assertion, stated and then used throughout, that the $z_i(\xx)$ are (unique and) necessarily \emph{lie in the interior of their intervals $I_i$}, moreover, being interior maximum locations, they satisfy $(P_i)'_t(\xx,z_i)=0$, where $P_i=L|_{I_i}$ is the weighted polynomial which agrees with the restriction of the Lebesgue function to $I_i$. In \cite{K-AMH}, the following statement appears (using the notation introduced here for simplicity): ``We correspondingly define $m_{n+1}$ to be the rightmost maximum value of the function $P_{n+1}$ and $z_{n+1}$ to be the point at which this rightmost maximum occurs (note that $z_{n+1}>x_n$ is not guaranteed by the definition, nor by the inherent nature of the functions being used). With these definitions completed, we note that $\|L\|_{\infty}=\max\{m_1,\dots,m_{n+1} \}$ and $P_i'(\xx,z_i)=0$ for every $i=1,\dots,n+1$.''
Below, we give an example where neither of these assertions holds.

\begin{example}
Let $n=2$ and take $x_0=0$, $x_1=1$ and $x_2=4$ be the interpolation nodes. Then we find that the Lebesgue function on the third interval $[4,\infty)$, denoted by $P_{n+1}=P_3$, has a maximum just at the left endpoint $4$. Furthermore, the derivative $(P_3)'_t(\xx,4)$ is strictly negative. All this can be explicitly calculated as follows. Consider the fundamental interpolation functions $h_k  \in Y_n$, satisfying $h_k(x_j)=\delta_{kj}$. Since this definition will also be needed for general $n$, we formulate it as follows, 
\begin{equation}\label{eq_hkt}
h_k(t):=h_k(w,\xx;t):=  \frac{e^{-t}}{e^{-x_k}}\prod\limits_{l=0, l\ne k}^n \frac{t-x_l}{x_k-x_l}, \qquad k=0,\dots, n.
\end{equation}
With the notation $\varepsilon_{k,i}:=(-1)^{k-i+1+\chi_{i\le k}}$, it leads to the formula 

\begin{equation*}
P_i(t):=P_{i}(\xx,t):=L(\xx,t)|_{I_i(\xx)}=e^{-t}\sum\limits_{k=0}^n \frac{\varepsilon_{k,i}}{e^{-x_k}}\prod\limits_{\substack{l=0, l\ne k}}^n\frac{t-x_l}{x_k-x_l}, \qquad i=1,\dots, n+1.
\end{equation*}
Now, rewriting the expression for the special case $n=2$ and differentiating of $P_{n+1}=P_3$ with respect to $t$ we obtain

\begin{equation}
\begin{aligned}
&(P_{3}(\xx,t))'_t=\frac{\partial}{\partial t} \left[e^{-t}\sum\limits_{k=0}^2 \frac{\varepsilon_{k,3}}{e^{-x_k}}\prod\limits_{l=0, l\ne k}^2\frac{t-x_l}{x_k-x_l} \right]\\
=-e^{-t}\sum\limits_{k=0}^2 &\frac{\varepsilon_{k,3}}{e^{-x_k}}\prod\limits_{\substack{l=0\\ l\ne k}}^2 \frac{t-x_l}{x_k-x_l}+e^{-t}\sum\limits_{k=0}^2 \frac{\varepsilon_{k,3}}{e^{-x_k}}\sum\limits_{\substack{m=0\\ m\ne k}}^2 \frac{1}{x_k-x_m}\prod\limits_{\substack{l=0\\ l\ne k,m}}^2 \frac{t-x_l}{x_k-x_l}.
\end{aligned}
\end{equation}
Substituting $t=4$, and $\xx=(x_0,x_1,x_2)=(0,1,4)$, we are led to $$\left.(P_3(\xx,t))'_t\right\vert_{t=4}=-1+e^{-4}\cdot\frac{3}{4}+\frac{e^{-4}}{e^{-1}}+\frac{1}{4}+\frac{1}{3}\approx -0.3531,$$ which is strictly negative. See Figure 1.
\end{example}

\begin{figure}[ht]
\includegraphics[width=11cm]{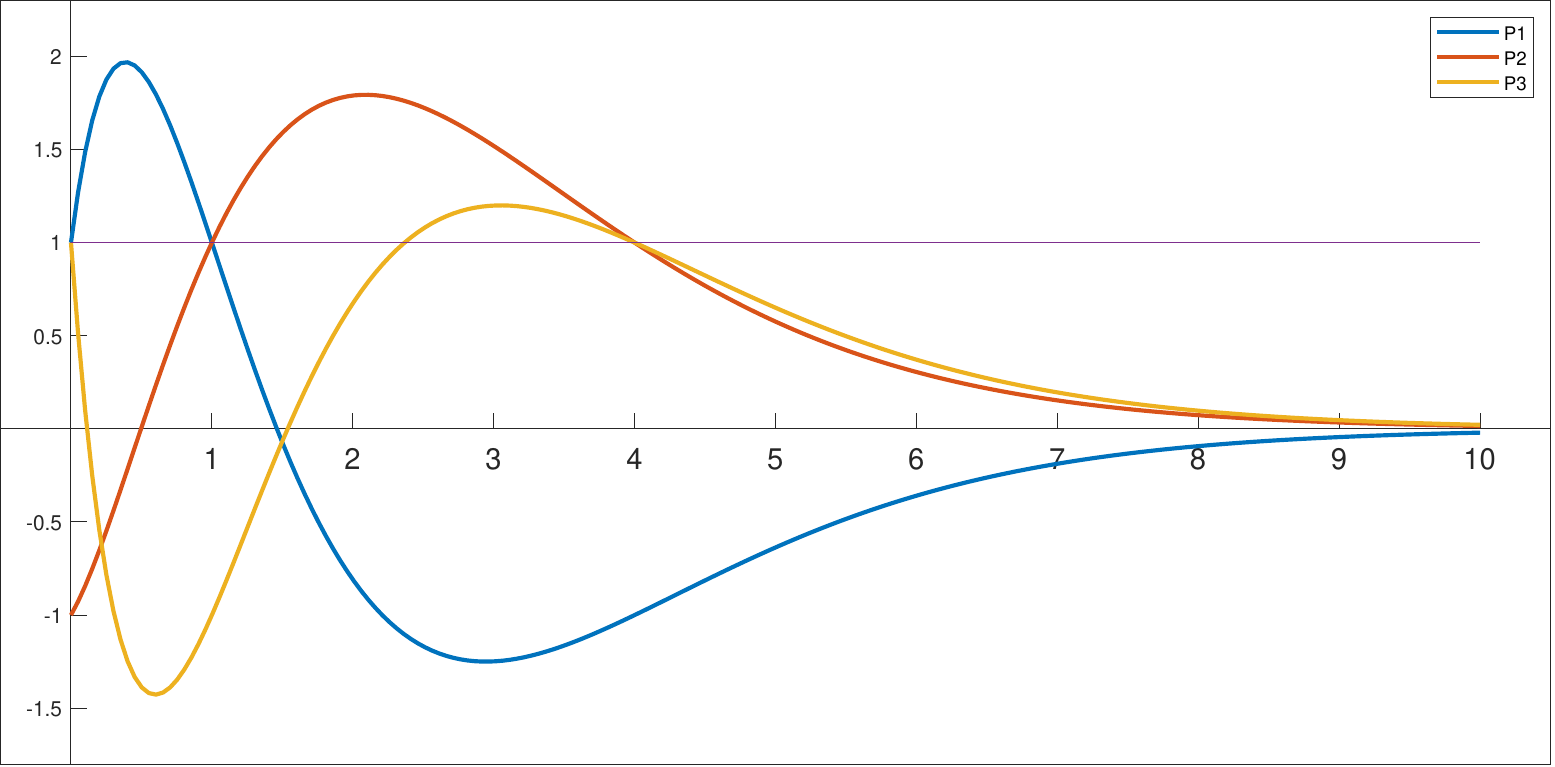}
\caption{}
\end{figure}

This means two things. First, it does not\footnote{Actually, the author himself records a warning about this problem in lines 5--6 of page 86 of \cite{K-AMH}, but later seems to forget about it and treats $z_{n+1}$ as interior to $(x_n,\infty)$, taking it for granted that, e.g., $(P_{n+1})'_t(z_{n+1})=0 $.} hold any more that $z_{n+1}$, i.e., the maximum location, is interior to its interval. Second, the formula (4) in \cite{K-AMH}, i.e., with the present notation $(P_i)'_t(\xx,z_i)=0$, ($i=1,\dots,n+1$) also fails in this case. It is crucial, because only (4) in \cite{K-AMH} and interior location of $z_{n+1}$ allows for the calculation of the derivatives according to (5) in \cite{K-AMH}, i.e., $\partial m_i(\xx)/\partial x_j=- h_j(\xx,z_i)(P_i)'_t(\xx,z_j)$. However, in this case continuity of $(P_{n+1})'_t(\xx,x_n)$ with respect to all nodes implies that $(P_{n+1})'_t(\xx,x_n)<0$ even for a neighborhood of the said node system $(0,1,4)$. Note that $P_{n+1}$ is a weighted polynomial (weighted by $\exp(-t)$) with alternating signs at the $x_i$, so that it has a maximal number of $n$ roots $\xi_i$, $i=1,\dots,n$ already in $(x_0,x_n)$, one in each $(x_{i-1},x_i)$, and hence it cannot have more. Consequently, its derivative has $n-1$ roots strictly between these roots. Finally, in view of the Lagrange mean value theorem, the derivative $(P_{n+1})'_t$ must be positive at some point $\eta \in (\xi_n,x_n)$, as $P_{n+1}(\xx,\xi_n)=0$ and $P_{n+1}(\xx,x_n)=1$. So for this point $\eta \in (\xi_n,x_n)$ we have $(P_{n+1})'_t(\xx,\eta)>0$, whereas $(P_{n+1})'_t(\xx,x_n)<0$, so that there is one more root $\zeta$ with $(P_{n+1})'_t(\xx,\zeta)=0$ between $\eta$ and $x_n$; again, we identified a maximal number of $n$ roots already in $(x_0,x_n)$ for $(P_{n+1})'_t$, hence there is no more, and $(P_{n+1})'_t<0$ all over $[x_n,\infty)$. It follows that $m_{n+1}(\xx)=P_{n+1}(\xx,x_n)=1$ identically for the said neighborhood of the node system $(0,1,4)$. Consequently, all partial derivatives $\partial m_{n+1}(\xx)/\partial x_j=0$, and the matrix $A$ has a vector ${\bf 0}$ in its last row and hence all $A_k$ are singular, save $A_{n+1}$.

A closer inspection of our previous example shows that for the node system $\xx=(0,1,4)$ one can find another system $\yy$ that violates the intertwining property. For instance, let 
$\yy=(0,1,5)$. Examining the maxima of the functions $P_i(\xx,t)$ and $P_i(\yy,t)$, we find that $m_i(\xx)\le m_i(\yy)$, $i=1,2,3$, which clearly contradicts the intertwining property. More precisely, the corresponding maxima take the following values:
\[
m_1(\xx)\approx 1.969<m_1(\yy)\approx 2.4, \quad m_2(\xx)\approx 1.794 < m_2(\yy)\approx 2.718, \quad m_3(\xx)=m_3(\yy)=1.
\]
See, the picture below, where $L_1:=L(\xx,t)$ and $L_2:=L(\yy,t)$.

\begin{figure}[ht]
\includegraphics[width=11cm]{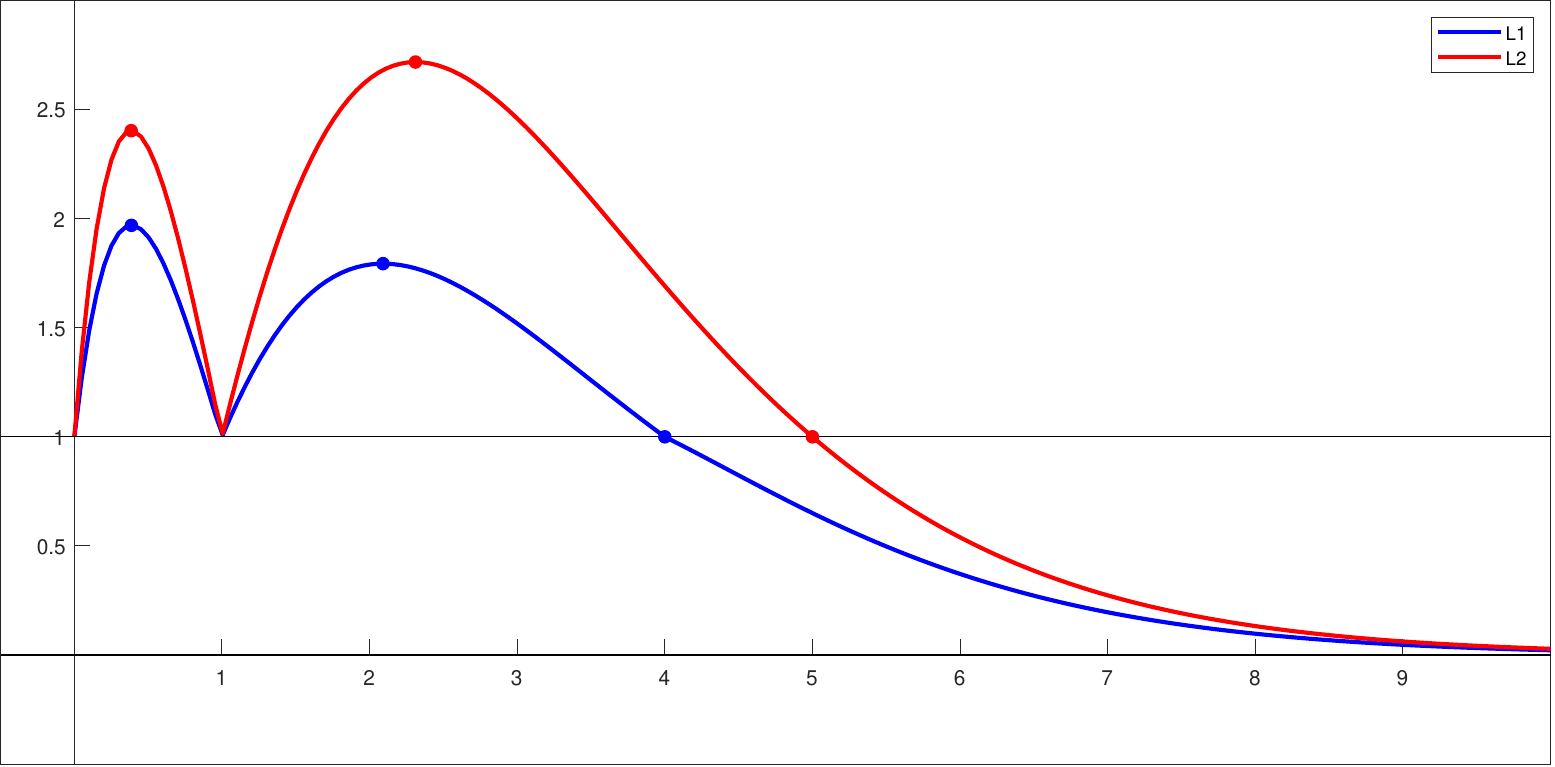}
\caption{}
\end{figure}

Let us remark here that it is not proved in \cite{K-AMH} that the other $z_j$ are unique, or that they stay interior to their intervals; also, there is no proof why $m_i$ would be differentiable with respect to the $x_j$. As long as the $z_i$ are (unique and) interior to their intervals, formula (4) holds, and the implicit function theorem can be applied to deduce continuous differentiability, but when a $z_j$ arrives at the endpoint, there is no guarantee that it remains continuously differentiable. It indeed \emph{seems} that for intervals $I_i$, enclosed by interpolation nodes at both ends, the respective Lebesgue function goes above 1 and hence the maximum point is not at any endpoint; but even if holding, it needs a proof\footnote{Also, the key formula (5) of the paper \cite{K-AMH} is not proved (it is called merely an observation). The closest to it in the literature were Lemma 2 in \cite{Kilgore-Cheney} and the Lemma on page 74 of \cite{Morris-Cheney}. As neither of them gives exactly what we are after, a more precise reference or a deduction from existing knowledge would be necessary. Note that the key papers \cite{Kilgore} and \cite{Kilgore-BullAMS} mention this only as a valid formula without giving any hint as to where it comes from; in the Conclusion section of \cite{Kilgore} the author refers to \emph{an earlier version of \cite{Kilgore-Cheney}}, which also seems meaning that in the public version it is missing.}.

\subsection{An example in case of a general weight}\label{sec:inner}

Whatever the situation may be with the weights $e^{-t}$ and $e^{-t^2}$, it is certainly not true for an arbitrary weight that the points $z_i$ lie in the interior of the corresponding intervals $I_i$ between interpolation nodes. Below, we construct a counterexample showing that  this property may well fail. 

\begin{example}
Let $\II:=[0,\infty)$, $(x_0,x_1)=(0,1)$ and $w(t):=\exp(-\sqrt{t})$. Then, $
h_0(t)=\exp(-\sqrt{t}) (1-t)$, $h_1(t)=\exp(1-\sqrt{t})t$, so that $P_1(\xx,t)= \exp(-\sqrt{t})[1+(e-1)t]$. It is easy to check that, conforming with the definitive properties, $P_1(\xx,0)=P_1(\xx,1)=1$. 

Differentiating, $(P_1)_t'(\xx,t)=\exp(-\sqrt{t})[e-1-\frac{e-1}{2}\sqrt{t}-\frac{1}{2\sqrt{t}}]$. Equating this to $0$ yields a quadratic equation whose roots are
$t_{1,2} = \left(1 \mp \sqrt{\frac{e-2}{e-1}}\right)^2 
\approx 0.12493 \quad \text{and} \quad 2.71112$.
Since the leading coefficient is negative, it follows that $P_1'(t) < 0$ for $t < t_1$ and for $t > t_2$, while $P_1'(t) > 0$ for $t_1 < t < t_2$.
Consequently, on $[0,1]$ the function $P_1$ attains its minimum at $t_1$, and its maximum is $1$, attained at the endpoints $t=0$ and $t=1$ (see Figure~3).
\end{example}

\begin{figure}[ht]
\includegraphics[width=11cm]{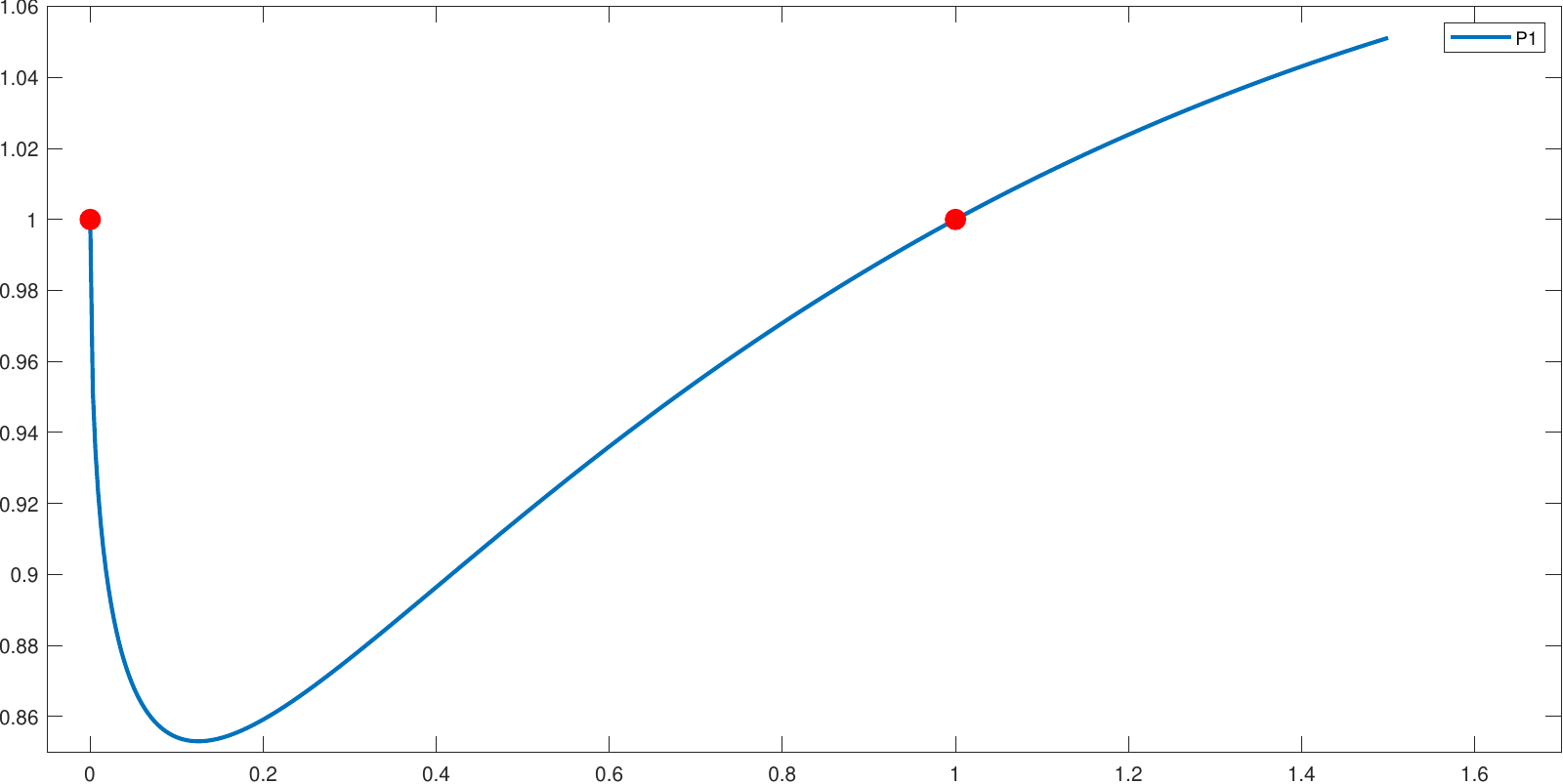}
\caption{}
\end{figure}

\subsection{Interpolation on the real line with Hermite weight}

Let us consider Hermite weights $w(t):=\exp(-t^2)$ on the real line together with weighted polynomials $Z_{n-1}:=\{pw~:~p\in\PP_{n-1}\}$. Theorem 4 of \cite{K-AMH} asserts that the Bernstein and Erd\H{o}s conjectures hold true also for interpolation in these weighted polynomial spaces, too.

However, similar problems as described in Section \ref{sec:exp} pertain to the case of the proof of Theorem 4 of \cite{K-AMH}, doubled by the fact that in that setup we have \emph{two infinite end intervals} not closed by interpolation nodes. So the same nullity of the gradient $\nabla m_i$ with respect to $\xx$ may (and does) occur even simultaneously for $i=1$ and $i=n+1$ for certain node systems. Then neither of $A_k$, nor even $A_{n+1}$ or $A_1$ (where at least one of the zero rows is omitted) can remain non-singular.  A counterexample for $n=3$ is as follows.

\begin{example}
Take $x_1=-c, x_2=0, x_3=c$, where $c:=\sqrt{C}$, and $C$ is the largest root of the equation $g(u)=\exp(-u)(4u-1)=1$. This function $g(u)$ is $-1$ at $u=0$, then it increases until its maximum place $u=5/4$, and then it decreases and tends to $0$ as $u\to \infty$. At its maximum $g(5/4)>1$, hence there are two positive roots $0<D<5/4<C$ of the equation; we take $C$ to be the larger one. Now with the above choice of $\xx$, the even function $f(t):=g(t^2)=\exp(-t^2)(4t^2-1)$ belongs to the space $Z_{n-1}$ of $\exp(-t^2)$-weighted polynomials of degree not exceeding $n-1=2$, and it has the property that $f(x_1)=1$, $f(x_2)=-1$ and $f(x_3)=1$. Given that $Z_{n-1}$ is a Chebyshev-Haar system, that means $f=P_1=P_4$, the end intervals' Lebesgue functions.  But, as is clear from the construction, these interval Lebesgue functions have maximums at their endpoints $x_1$, respectively $x_3$, with maximum value $1$, and everywhere in their interval $I_1=(-\infty,x_1]$ resp. $I_4=[x_3,\infty)$ they are strictly less than 1. Moreover, $(P_1)'_t(\xx,-c)>0, (P_4)'_t(\xx,c)<0$ (see Figure 4, where $d:=\sqrt{D}$). As before, it follows that $\nabla m_1 = \nabla m_4={\bf 0}$.
\end{example}

\begin{figure}[ht]
\includegraphics[width=11cm]{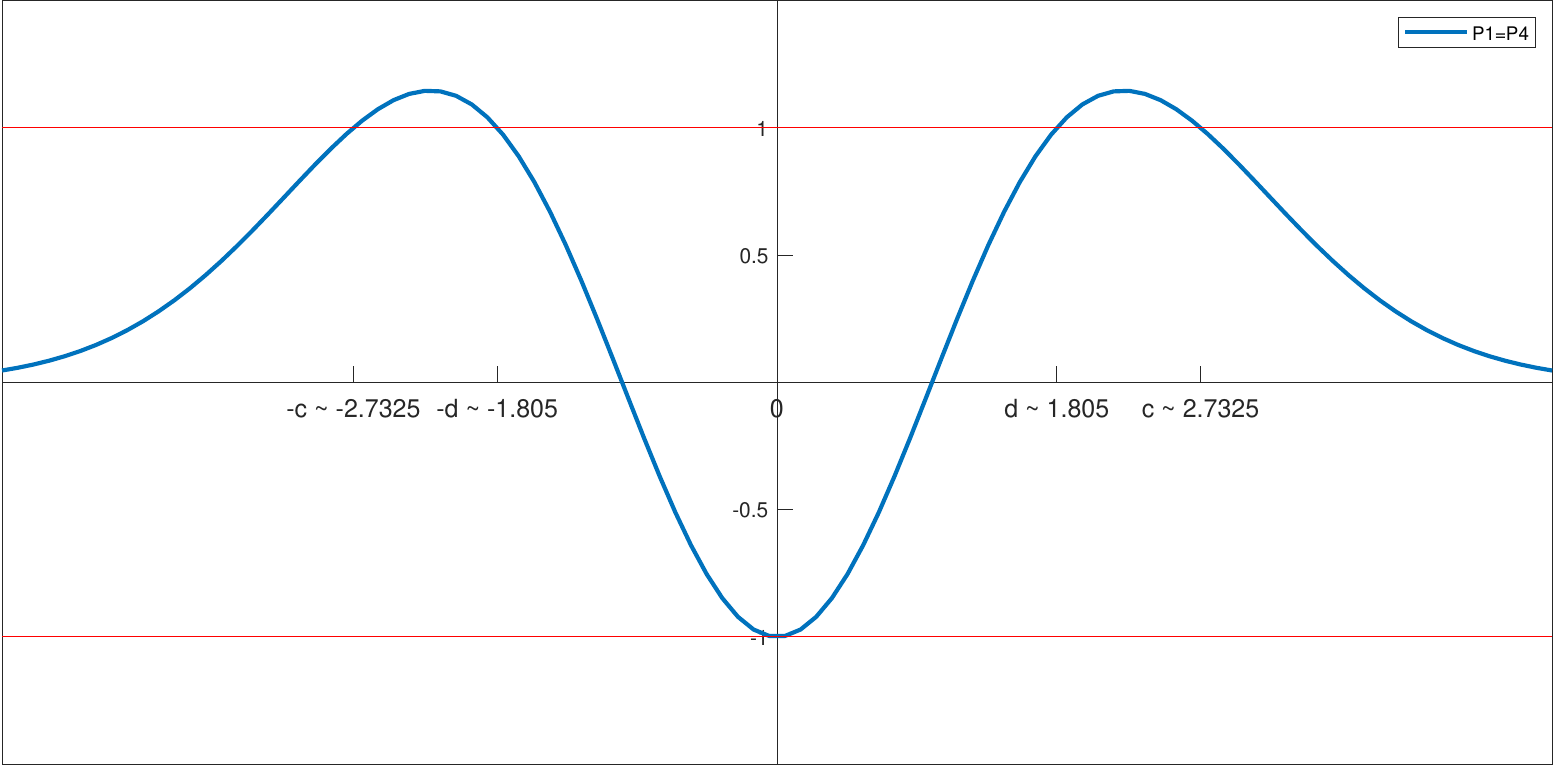}
\caption{}
\end{figure}

\section{Interpolation at infinity and "hybrid systems"}\label{sec:hybrid}

In the present section, we are going to introduce the precise terminology for the interpolation with the hybrid system $\YY_n$ and provide a detailed proof of the interlacing property of the roots of $P_i$ and the Markov-type inheritance theorem for $\YY_n$. As a consequence, we also obtain the Bernstein and Erdős theorem for the hybrid system. 

\subsection{Basics for the hybrid system}

First, we show that $\YY_n$ forms an extended Chebyshev–Haar system. A set of functions $\mathcal{F}$ is called a \emph{rank $n$ Extended Chebyshev-Haar system (or ECH system, ECHS)}, if any element of $\mathcal{F}$ either has at most $n$ real zeros (with their multiplicities counted) or is identically zero. (For more details, see \cite{Haar}, \cite{KarlinStudden} or \cite{MN}.) 

\begin{lemma}
The set $\YY_n:={\rm span}\{Y_n, {\bf 1}\}$ is a rank $n+1$ ECHS. 
\end{lemma}

\begin{proof}
Assume that $f\in \YY_n$ has at least $n+2$ roots with multiplicities. Then, $f'\in \WW_n$, i.e., it belongs to the rank $n$ ECHS $\WW_n$ whereas it has at least $n+1$ zeros, which is only possible if $f'$ is constant zero, i.e., $f$ is also constant, whence constant zero. 
\end{proof}

Throughout this paper, we shall refer to the elements of $\YY_n$ as hybrid polynomials. Let $f\in \YY_n$, i.e.,  
\begin{equation*}
f(t)=e^{-t}p(t)+c=e^{-t}(at^n+p_{n-1}(t))+c,
\end{equation*}
with $p\in \Pn$, $p_{n-1}\in \Pnn$ and $c\in\RR$. Then, by the degree of $f$ we mean the degree of $p$, i.e., $\deg f:= \deg p$; $c$ is called the constant term of $f$, and we say $a$ is the leading coefficient of $f$. (If we would like to emphasize that $c$ and $a$ belong to the element $f$, we will use the notations $c(f)$ and $a(f)$, respectively.) So, if we omit the constant term of $f$ and multiply by $e^t$, we get an algebraic polynomial, where $a$ is the coefficient of the $n$th power term. Therefore, in the current terminology --unlike in the case of algebraic polynomials-- the leading coefficient of a hybrid polynomial in $\YY_n$ is defined as the coefficient of the $n$th power term in the polynomial part, which can thus be zero.

We recall that in the case of the hybrid system $\YY_n$, we consider $\II:=[0,\infty)$ and the set of node systems is $S:=\{\xx=(x_1,\dots,x_n)\in \RR^{n}: 0=x_0<x_1<\dots<x_n<\infty\}$. We introduce the intervals $I_k=[x_{k-1},x_k]$ for every $k=1,\dots,n$, whereas $I_0:=(-\infty, x_0]$, and $I_{n+1}:=[x_{n},\infty)$.
We denote the elements of the classical, unweighted Lagrange interpolation basis by $\ell_k(t)$, i.e., we take $\lk(t):=\lk(\xx;t):= \prod_{j=0,~j\ne k}^{n} \left(\frac{t-x_j}{x_k-x_j}\right)$ for every $k=0,\ldots,n$. 

Let $h_k$ ($k=0,\dots, n+1$) stand for the $k$th interpolation basis element of the hybrid system $\YY_n$, i.e., for $k=0,\dots, n$, $h_k$ is defined by the formula \eqref{eq_hkt}
and the last element of the interpolation basis $h_{n+1}$ of $\YY_n$ is defined by:
\[
h_{n+1}(t):=1-\sum\limits_{k=0}^n h_k(t).
\]
Then, we get that $h_k\in\WW_{n}$ ($k=0,\ldots,n$), whereas $h_{n+1}\in \YY_n$; and  $h_k(x_j)=\de_{jk}$ for every $k,j=0,\dots,n+1$, where $h_k(x_{n+1}):=\lim_{t\to\infty} h_k(t)$.

The weighted Lebesgue function for the hybrid system is defined by
$$
L(w,\xx,t):=L(\xx,t):=\sum_{k=0}^{n+1} \left|h_k(t)\right|= \sum_{k=0}^{n+1} \vk(t) h_k(t),
$$
where $\vk(t):=\sign h_k(t)$, i.e.,
$$
\vk(t):=\begin{cases} (-1)^{k-i} \qquad & \textrm{if} \quad t\in \intt I_i, ~ 0 \le  i \le k
\\ (-1)^{k+1-i} & \textrm{if} \quad t\in \intt I_i, ~ k < i \le n+1\\
1& \textrm{if} \quad t=x_k, k=1,\dots,n\\
0& \textrm{if} \quad t=x_i, i\ne k.
\end{cases}
$$
Using the notation $\vki:=(-1)^{k+1-i+\chi_{i\le k}}\quad (0\le i\le n+1)$, (with $\chi_{i\le k}=1$, if $i\le k$ and otherwise it is 0); and restricting $L$ to the subinterval $I_i$  we get a hybrid polynomial
\begin{equation*}\label{Pi}
P_i:=L|_{I_i}= \sum_{k=0}^n \vki h_k+\varepsilon_{n+1,i}h_{n+1} \in \YY_n \quad (1\le i\le n+1).
\end{equation*}
We need to examine the extension of these functions to the whole interval $\RR$. Let $k\equiv n$ denote, that $k$ is congruent to $n$ modulo 2. Then, we get  
\begin{equation}\label{Pi2}
P_i=\sum\limits_{k=0}^n \vki h_k +\varepsilon_{n+1,i}h_{n+1}=\varepsilon_{n+1,i}\left(1-2\sum\limits_{\substack{k=0, \\k\equiv n, \text{ if } k\ge i\\ k\not\equiv n, \text{ if }k<i}}^n h_k\right).
\end{equation}

By construction, the values of $P_i$ at the interpolation nodes $x_0,\dots, x_n, x_{n+1}$ are
\begin{align}\label{Pixm}
P_i(\xx,x_m):=P_i(x_m)=
\begin{cases}
(-1)^{m-i+1}, & \text{if } m<i,\\
(-1)^{m-i}, & \text{if }  m\ge i,
\end{cases} 
\end{align}
where $P_i(\xx,x_{n+1}):=\lim\limits_{t\to \infty} P_i(\xx,t):=\lim\limits_{t\to \infty} P_i(t)$.

\subsection{Interlacing property of $P_i$'s}

Our aim is to show, that the hybrid polynomials $P_i$ have the interlacing property. According to \eqref{Pixm} $P_i(x_{i-1})=P_i(x_{i})=1$ ($i=1,\ldots, n$), also $P_{n+1}(x_n)=\lim\limits_{t\to \infty}P_{n+1}(t)=1$.

Moreover, the function values of $P_i$ $(i=1,\dots, n+1)$ have alternate signs on the $n+2$ nodes except for $x_{i-1}$ and $x_{i}$ ($i=1,\dots, n+1$), where the signs are both equal to 1. Thus, on the $n+2$ nodes, we have $n+1$ alternating signs and $n$ zeros in between the alternating nodes. Further, these zeros are unique and simple in their intervals, for otherwise to meet the required signs at the endpoint nodes, there would have to be at least three zeros (counted with multiplicities), altogether the number of zeros climbing to at least $n+2$, which contradicts the fact that $P_i$ belongs to the rank $n+1$ ECH system $\YY_n$. As already told, $P_i(x_{i-1})=P_{i}(x_{i})=1$ ($i=1,\ldots,n+1$) and hence for similar reasons as before, there is no zero in $I_i$.

Let $a_i:=a(P_i)$, 
i.e., it is the coefficient of the $n$th power term
\begin{equation}\label{aidef}
-2 \varepsilon_{n+1,i}\sum\limits_{\substack{k=0, \\k\equiv n, \text{ if } k\ge i\\ k\not\equiv n, \text{ if }k<i}}^n e^{x_k}\lk,
\end{equation}
and $b_k$ denote the leading coefficients of the Lagrange fundamental polynomials $\lk$. They are
\begin{equation}\label{bk}
b_k=\prod_{\substack{j=0\\j\ne k}}^n \frac{1}{x_k-x_j}=(-1)^{n-k} |b_k|, \quad k=0,\dots, n.
\end{equation}
Now to compute $a_i$ is easy from formulas \eqref{Pi2} and \eqref{aidef}
\begin{equation}\label{ai}
\begin{aligned}
a_i&= -2\vnni\sum\limits_{\substack{k=0, \\k\equiv n, \text{ if } k\ge i\\ k\not\equiv n, \text{ if }k<i}}^n e^{x_k}(-1)^{n-k}|b_k|\\
&=-2\vnni\sum\limits_{\substack{k=0, \\ k\not\equiv n}}^{i-1} e^{x_k}(-1)^{n-k}|b_k|-2\vnni\sum\limits_{\substack{k=i, \\ k\equiv n}}^{n} e^{x_k}(-1)^{n-k}|b_k|\\
&=2\vnni \left(\sum\limits_{\substack{k=0, \\ k\not\equiv n}}^{i-1} e^{x_k}|b_k|-\sum\limits_{\substack{k=i, \\ k\equiv n}}^{n} e^{x_k}|b_k| \right).
\end{aligned}
\end{equation}
Here, the first sum is counted for indices $k$, with $k\not\equiv n$ and then $(-1)^{n-k}=-1$, and similarly, in the second sum $k\equiv n$, i.e., $(-1)^{n-k}=1$. 

Consequently, we have that
\begin{align}
(-1)^{n+i+2}a_{i+1}&-(-1)^{n+i+1}a_i=2\left(\sum\limits_{\substack{k=0, \\ k\not\equiv n}}^{i} e^{x_k}|b_k|-\sum\limits_{\substack{k=i+1, \\ k\equiv n}}^{n} e^{x_k}|b_k| \right)\\ \notag
&-2\left(\sum\limits_{\substack{k=0, \\ k\not\equiv n}}^{i-1} e^{x_k}|b_k|-\sum\limits_{\substack{k=i, \\ k\equiv n}}^{n} e^{x_k}|b_k| \right)=4|b_i|e^{x_i}.
\end{align}

From this, it is clear that 

\begin{equation}\label{ai_order}
a_{n+1}>\cdots > (-1)^{n+1+i}a_i > \cdots > (-1)^{n+2} a_1.
\end{equation}

In particular, the first term in the chain, $a_{n+1}$ is positive, but we also have that $(-1)^{n+2}a_1<0$ because from the explicit expression \eqref{bk} $|b_0|<|b_1| <|b_1|e^{x_1}$ is also seen to hold. There exists a certain index\footnote{Applying  $(-1)^{n+2}a_1<0$, it means, that if $r=1$, then $(-1)^{n+2} a_1 < 0$ and $(-1)^{n+1+q} a_q>0$ for $q>1$.} $1 \le r \le n$ such that $(-1)^{n+1+s} a_s \le 0$ for $s\le r$ (and is strictly $<0$ for $s=1,\ldots,r-1$), and then $(-1)^{n+1+q} a_q>0$ for $q>r$.

Recall that we have already found $n$ simple zeros of $P_i$ for each $i=1,\ldots n+1$ in each $\intt I_k$, $k=1,\ldots, n+1$, save the very $I_i$. Since the functions $P_1,\ldots,P_{n+1}$ belong to an ECHS of rank $n+1$, $P_i$ can have at most one further zero. If $a_r=0$, then $P_r$ cannot have more than $n$ zeros--we found them all (as it cannot be identically zero for $P_r(x_r)=1$). For the remaining functions $P_i$, we know that $a_i\ne 0$, and we can identify their signs towards $-\infty$ from the leading coefficient. Indeed, we have $\lim_{t \to -\infty} P_i(t)/(e^{-t}t^{n}) = (-1)^{n} a_i$, hence for small values $- T$ we have
$$\sign P_i ( -T) = (-1)^{n} \sign a_i,$$
and from \eqref{Pixm}, we get that
\begin{equation*}
\sign P_i(x_0)=\sign P_i(0)=(-1)^{1-i}.
\end{equation*}

Therefore, if $(-1)^{n+1+i} a_i <0$ (i.e., if $i<r$ and even for $i=r$ if $a_r \ne 0$), then $\sign P_i(-T)= (-1)^{n+n+1+i-1}=(-1)^i$, whereas $P_i(x_0)=(-1)^{1-i}$, and there is one more sign change in $(-T,0)$, with one more zero of $P_i$. Moreover, if $(-1)^{n+1+i} a_i >0$ (i.e., if $i>r$), then $\sign P_i(-T)= \sign P_i(0)=(-1)^{1-i}$ and therefore, there are no more zeros in $(-\infty,0)$.

Altogether, we identified zeros of the hybrid polynomials $P_i$; for $i$ from $1$ up to $r-1$ they have $n+1$ zeros, such that one of their zeros are in $(-\infty,x_0)$, for $i=r$ either we have the same (if $a_r\ne 0$) or $P_r$ has $n$ roots only (if $a_r=0$), and for $i$ from $r+1$ up to $n+1$ they have $n$ simple real zeros as many as their degree. 

According to the above, we now define the index set $\Ji$ to contain all indices $k$ for which $ P_i$ has a root in $I_k$.  
Then we will have\footnote{In particular, $(-1)^{n+2}a_1<0$ guarantees that $0\in \mathcal{J}_1$.}
\begin{equation}\label{Jistardef}
\Ji=\begin{cases}
 \{0,1,\ldots,n+1\}\setminus \{i\} & \textrm{if} \quad 1\le i<r
\\  & \textrm{or} \quad i=r \quad \textrm{and} \quad {a_r} \ne 0;
\\ \{1,\ldots,n+1\}\setminus \{i\} & \textrm{if} \quad i=r \quad \textrm{and} \quad {a_r} = 0
\\  & \textrm{or} \quad  r<i\le n+1.
\end{cases}
\end{equation}

\subsection{Interlacing of roots of $P_i$} This part of the analysis follows the approach of \cite{CBoorPinkus} (see Lemma~1 therein), replacing algebraic polynomials with hybrid ones, whose zeros exhibit a somewhat different structure.  More precisely, the structure of the distribution of their zeros (i.e., the sets $\Ji$) has a different form. 

The following simple lemma will be used repeatedly in our argument.

\begin{lemma}\label{l:simple} Let $(\alpha,\beta]$ be any half-open half-closed interval and $\ff, \psi : (\alpha,\beta] \to \RR$ be continuous functions with $\psi(\beta)>0$. Denote $z$ any zero of $\ff$ and $w$ the largest zero of $\psi$.

Then if $f:=\ff - \psi >0$ on $(\alpha,\beta)$, then we have $z<w$.

In particular, if $w$ does not exist, that is, $\psi\ne 0$ on $(\alpha,\beta]$, then $z$ can not exist either, that is, $\ff \ne 0$. 
\end{lemma}

\begin{proof} Obviously, $w$ being the largest zero of $\psi$ and $\psi(\beta)>0$ implies that $\psi>0$ on the whole of $(w,\beta)$. As $\ff-\psi>0$ on $(\alpha,\beta)\supset (w,\beta)$, we must have here $\ff >\psi >0$, hence $z$ cannot belong to this interval. Further, $z=w$ is excluded, too in view of $\ff(w) >\psi(w)=0$. 

Obviously, if $w$ does not exist, then $\ff >\psi>0$ and $z$ cannot exist either.
\end{proof}

To deal with zeros of different $P_i$, let $y_k^{(i)}$ denote the unique zero of $P_i$ in $I_k$; we have such a zero precisely for $k \in \Ji$. Further, denote $\Kij:=(\Ji \cap \Jj) \setminus \{i+1,\ldots,j-1\}$ if $i<j$.

\begin{lemma}\label{l:zerosinsame} Let $1 \le i < j \le n+1$ be arbitrary. If $k \in \Kij$, then the zeros of $P_i$ and $P_j$ satisfy $y_k^{(j)} < y_k^{(i)}$; and if $i<k<j$, then $y_k^{(i)} < y_k^{(j)}  $. \end{lemma}
Note that $\{i+1,\ldots,j-1\} \subset \Ji \cap \Jj$ and hence the two index sets in the conditions sum up precisely to $\Ji \cap \Jj$.

\begin{proof}
 We will consider two auxiliary functions and consider their sign changes and zeros. Namely, let for any $1\le i<j\le n+1$
\begin{align*}
G(t):= (-1)^{i+1} P_i(t)+(-1)^{j+1} P_j(t), \\
H(t):= (-1)^{i+1} P_i(t)-(-1)^{j+1} P_j(t).
\end{align*}

Applying \eqref{Pixm} we know the values of $G$ at the nodes $x_0,\ldots, x_n$, 
\begin{align}\label{gxk}
 G(x_m) = (-1)^{i+1} P_i(x_m)+(-1)^{j+1} P_j(x_m)
=\begin{cases}
2(-1)^{m+1}, & \text{ if } i<j\le m\\
2(-1)^m, & \text{ if } m<i<j,\\
0, & \text{ otherwise}.
\end{cases} 
\end{align}
By \eqref{Pixm}, we get that $\lim_{t\to \infty} P_i(t)=(-1)^{n-i+1}$, for every $i=0,\dots, n$, which implies that $\lim_{t\to \infty} G(t)=2(-1)^n$ and $\lim_{t\to \infty} H(t)=0$. 

Similarly, as in the case of $G$, for every $m=0,\dots,n$, we get that 
\begin{equation}\label{hxk}
\begin{aligned}
H(x_m) = \begin{cases}
2(-1)^{m+1}, & \text{ if } i\le m<j,\\
0, & \text{ otherwise}.
\end{cases}
\end{aligned}
\end{equation}

Let us consider first the function $G$. We know that it is an element of $\YY_n$, and hence $G$ can have at most $n+1$ roots (if the leading coefficient $a(G)\ne 0$), or at most $n$ (if $a(G)=0$). 

From \eqref{gxk} we get that if $j=n+1$ and $i=n$, then $\sign G(x_{m-1})=-\sign G(x_{m})$, for every $m=1,\dots,n-1$, so $G$ has $n$ strictly alternating signs, and also $G(x_n)=0$. Altogether, it yields $n-1+1=n$ zeros of $G$. 

If $j=n+1$ and $i<n$, then $\sign G(x_{m-1})=-\sign G(x_{m})$ for $m=1,\dots,i-1$, which implies $i-1$ roots of $G$, and besides them there are $n-i+1$ zeros of $G$ at $x_i, \dots, x_n$. Altogether, it also means $n$ zeros of $G$.

Otherwise, for $j<n+1$ $\sign G(x_{m-1})=-\sign G(x_{m})$, if  $m\in \{1,\dots, i-1\}\cup \{j+1,\dots,n\}$ and $G(x_m)=0$ for $m=i,\dots, j-1$. Therefore, we obtain $i-1+n-j$ zeros in $\intt I_1, \ldots \intt I_{i-1}, \intt I_{j+1},\dots,\intt I_{n}$, and $j-i$ zeros exactly at $x_{i},\ldots,x_{j-1}$. Altogether, this amounts to $n-1$ zeros.
However, we can apply \eqref{gxk}, which implies that $G(x_n)=2(-1)^{n+1}$, whereas $\lim_{t\to \infty} G(t)=2(-1)^n$. Then there has to be an additional zero of $G$, in $(x_n,\infty)=I_{n+1}$. 

Consequently, for every $1\le i<j\le n+1$ we found $n$ zeros of $G$. 

In the following argument, we need the fact that $G$ has precisely $n$ zeros in the half-
line $(0,\infty)$ irrespective of the various subcases considered above. Indeed, we have found $n$ zeros in all cases, and $G(x_0)=2$ according to \eqref{gxk}, whereas $\lim_{t\to \infty} G(t)=2(-1)^n$, so the number of zeros in $(0,\infty)$ must be of the same parity as $n$ and therefore is $n$. In particular, $G$ has constant sign in the interior of each interval $I_k$ for $k=i+1,\dots, j-1$ (and in fact also for $k=i$ and $k=j$).

So let now $I_k=[x_{k-1},x_{k}]$ be any of the intervals with $i<k<j$. We know that there are sign changes and hence zeros of $P_i$ and $P_j$ in $I_k$. Also, $G(x_{k-1})=G(x_{k})=0$, and $G$ has no more zeros in $I_k$, so that it has constant sign on $(x_{k-1},x_{k})=\intt I_k$. 
Our aim is to get the sign of $G$ on $\intt I_k$. Firstly, \eqref{gxk} implies that $G(x_{i-1})=2(-1)^{i-1}$
 and then the sign of $G$ is $(-1)^{k-i}(-1)^{i-1}=(-1)^{k-1}$ 
on $\intt I_k$.
Consider now, $f:=(-1)^{k-1}G$ 
on $I_k$. It is zero at the endpoints $x_{k-1}$ and $x_k$ and positive inside the interval. Moreover,
\begin{equation*}
f=(-1)^{k-1}G=(-1)^{k+i}P_i+(-1)^{k+j}P_j=(-1)^{k+i}P_i-(-1)^{k+j+1}P_j
\end{equation*}
and applying \eqref{Pixm} (and $k<j$), we get that

\begin{equation*}
(-1)^{k+j+1}P_j(x_k)=(-1)^{k+j+1}(-1)^{k-j+1}>0.
\end{equation*}

\bigskip

Now as $\ff:=(-1)^{k+i}P_i$ 
and $\psi:=(-1)^{k+j+1}P_j$
satisfy the conditions of the simple Lemma \ref{l:simple} for $I_k$, and they both have only one zero $y^{(i)}_k$ and $y_k^{(j)}$ (the zeros of $P_i$ and $P_j$) in $I_k$, respectively, we are led to $y^{(i)}_k<y_k^{(j)}$, as wanted.

\medskip

For the remaining $k \in \Kij$, we argue similarly using $H$. We note, that the constant term of $H$ is $c(H)=0$, i.e., $H\in \WW_n$ and hence $\deg H\le n$. We note that this part proceeds analogously to the polynomial case, as in Lemma 1 of de Boor and Pinkus \cite{CBoorPinkus}. However, for the sake of completeness, we now present the proof in the present setting.

We show that $\deg H=n$. Let us first consider the case when $i+1<j<n+1$. Applying \eqref{hxk} we have that $H(x_0)=\dots=H(x_{i-1})=0$, $H(x_{j})=\dots=H(x_n)=0$, while $H(x_m)=2(-1)^{m+1}$ for $m=i,\ldots,j-1$, giving a strict alternation of signs on $j-i$ points and hence $j-i-1$ zeros in $(x_{i},x_{j-1})$. Altogether, these amount to $i+ n-j+1 + j-i-1= n$ zeros, thus we have identified a maximal possible number of zeros for the weighted polynomial $H$, all simple, with no further ones existing. Similarly, one can check easily, that $H$ has also $n$ zeros in the particular cases when $i+1=j\le n+1$ (namely, at $\{x_0,\dots,x_n\}\setminus\{x_i\}$); and when $i<n<j=n+1$ ($i$ zeros at $x_0,\dots, x_{i-1}$ and $n+1-i-1=n-i$ zeros in $(x_i,x_n)$). This entails that $H$ has a constant sign between these found zeros, and also before the first and after the last zero. 

Since $H(x_i)=2(-1)^{i+1}$, we also get that $\sign H=(-1)^{i+1}(-1)^{k-i}=(-1)^{k+1}$ on $I_k$ for $k=0,\dots, i-1$ and hence $(-1)^{k+1}H>0$ on $\intt I_k$ for all $k=0,\ldots,i-1$. Similarly, $H(x_{j-1})=2(-1)^j$ and it implies that $\sign H=(-1)^{k}$ on $I_k$ for $k=j+1,\dots,n+1$.

Regarding the endpoints, we apply \eqref{Pixm} and find
\begin{align*}
&(-1)^{i+1}P_i(x_{k})=(-1)^{j+1}P_j(x_{k})=(-1)^{k}, 
\qquad (k=0,\ldots, i-1),
\\[1mm]
&(-1)^{i+1}P_i(x_{k})=(-1)^{j+1}P_j(x_{k})=(-1)^{k+1}, 
\quad (k=j+1,\ldots, n+1).
\end{align*}
First let us consider $I_k$ with $0 \le k<i$ and put $f:=(-1)^{k+1} H$, $\ff=(-1)^{k+j+1}P_j$ and $\psi:=(-1)^{k+i+1}P_i$, so that $f=\ff-\psi>0$ on $\intt I_k$. We have for the right endpoint $\psi(x_{k})= (-1)^{k+i+1}P_i(x_{k})=(-1)^{k+i+1}(-1)^{k+i+1}=1>0$ for $k<i$, hence Lemma \ref{l:simple} furnishes that the root of $P_j$ comes before the root of $P_i$, that is, $y_k^{(j)}<y_k^{(i)}$ for $k=1,\ldots,i-1$, and, if there are zeros of both in $I_0$, then also for $k=0$, since $\psi(x_0)>0$ and $f=\ff-\psi>0$ on $\intt I_0$. Here we have zeros for both if and only if $0 \in \Ji$ and $0\in \Jj$, that is, precisely when $0 \in \Kij$.

Now let us consider the case $j<k \le n+1$. Now we put $f:=(-1)^{k}H$, $\ff:=(-1)^{k+j}P_j$ and $\psi:=(-1)^{k+i}P_i$, so that $f:=\ff-\psi >0$ on $I_k$. Further, now $\psi(x_{k})=(-1)^{k+i} P_i(x_{k})=(-1)^{k+i+k+i+2}=1>0$ ($j<k\le n$) and for any $j<k<n$  Lemma \ref{l:simple} applies again. Therefore, the inequality $y_k^{(j)}<y_k^{(i)}$ for $k=j+1,\ldots,n$ obtains. Further, if $n+1 \in \Ji, \Jj$, then there are roots of both hybrid polynomials $P_i$ and $P_j$. In this case we indeed see that $\lim_{t \to +\infty} \psi(t) = \lim_{t \to +\infty} (-1)^{n+1+i} P_i(t) = \lim_{t \to +\infty} (-1)^{n+1+i} (-1)^{n+1+i} = +1$. Therefore, for some large enough $T$ with $\psi(T)>0$ we can apply Lemma \ref{l:simple} once more to infer the same result also for $k=n+1$.
\end{proof}

\begin{lemma}\label{l:interlacing} For any $1\le i < j \le n+1$, the zeros of $P_i$ and $P_{j}$ strictly interlace, that is, strictly between any two consecutive (simple) zeros of one of the hybrid polynomials, there is exactly one (simple) zero of the other, too. Moreover, in the cases where $i<r<j$, or $i=r<j$ with $a_r\ne 0$, or $i<j=r$ with $a_r=0$, we have $0\in \Ji$ and $0\notin \Jj$. Hence, the chain of interlacing roots begins with the leftmost root of $P_i$; in all remaining cases, it starts with the first root of $P_j$.
\end{lemma}

\begin{proof} Consider first the case when both $P_i, P_j$ have a zero in $I_0$. This is the case when $ i<j< r(\le n)$, and if $a_r\ne 0$, then also if $i<j=r(\le n)$. Applying Lemma \ref{l:zerosinsame} to $i, j$ we get that $y_0^{(j)}<y_0^{(i)}$, i.e., the first root of $P_j$ precedes the roots of $P_i$. We also have zeros of both hybrid polynomials $P_i$, $P_j$ for all $I_k$, $(1\le k\le n+1, k\ne i,j)$ and  Lemma \ref{l:zerosinsame} guarantees again that:
\begin{equation}\label{roots}
\begin{aligned}
&y_k^{(j)}<y_k^{(i)}, \quad \textrm{for all} \quad 1\le k \le i-1 \text{ and } j+1<k\le n+1\\
&y_k^{(i)}<y_k^{(j)}, \quad \textrm{for all} \quad i+1\le k \le j-1.
\end{aligned}
\end{equation}
However, concerning the intervals $I_i$ and $I_j$, we know that $y_{i}^{(i)}$, respectively $y_{j}^{(j)}$ do not exist. Therefore, depending on the values of $i$ and $j$, we have one of the following chains of inequalities (around the intervals $I_i$ and $I_j$):
\begin{equation}\label{roots_around_xi}
\left.
\begin{aligned}
y_{i-1}^{(j)}&<y_{i-1}^{(i)}<x_{i-1}<y_{i}^{(j)}<x_{i} < y_{i+1}^{(i)} < y_{i+1}^{(j)}< \dots \\
\dots&<x_{j-1}<y_{j}^{(i)}<x_{j}<y_{j+1}^{(j)} \end{aligned} \right\} \quad \text{ if $j>i+1$, }
\end{equation}
\begin{equation}\label{roots_around_xi2}
y_{i-1}^{(j)}<y_{i-1}^{(i)}<x_{i-1}<y_{i}^{(j)}<x_{i} < y_{i+1}^{(i)} < x_{i+1}< y_{i+2}^{(j)}, \quad \text{ if } j=i+1.
\end{equation}
So we get the desired interlacing property for the roots of $P_i$ and $P_j$.

Assume now, that $1\le i<j=r(\le n)$ such that $a_r=a_j=0$. Then, according to \eqref{Jistardef} we get that $\Ji=\{0,\dots,n+1\}\setminus\{i\}$, and $\Jj=\{1,\dots,n+1\}\setminus\{j\}$. Therefore, applying Lemma \ref{l:zerosinsame}, we get that the inequalities listed in \eqref{roots} and the corresponding chain(s) of inequalities from \eqref{roots_around_xi} and \eqref{roots_around_xi2} are satisfied.
Moreover, since $0\in \Ji$, but $0\notin \Jj$, we also have
\begin{equation}\label{beginning}
y_0^{(i)}<x_0<y_1^{(j)},
\end{equation}
and hence the chain of the interlacing roots begins with the first root of $P_i$.
Altogether, in this case, we also get the interlacing property of $P_i$ and $P_j$ on $\RR$, too.

Now, consider the cases $i< r$ or $i=r$ such that $a_r\ne0$ and $r<j\le n+1$. Then, $0\in \Ji$ and $0\notin \Jj$, and each inequality in \eqref{roots} and \eqref{beginning}, moreover the corresponding chain(s) of inequalities from \eqref{roots_around_xi} and \eqref{roots_around_xi2} are satisfied. We recall that $n+1\notin \Jj$ if and only if $j=n+1$. Moreover, then the second inequality of \eqref{roots} holds for every $k\in\{i+1,\dots,n+1\}$.  Therefore, concerning the last two intervals  $I_n, I_{n+1}$, we get one of the following chains of inequalities
\begin{equation}\label{ending}
\begin{aligned}
&x_{n-1}<y_{n}^{(j)}<y_{n}^{(i)}<x_n<y_{n+1}^{(j)}<y_{n+1}^{(i)}, \quad &\text{ if } j<n,\\
&x_{n-1}<y_{n}^{(i)}<x_n<y_{n+1}^{(j)}<y_{n+1}^{(i)}, \quad &\text{ if } j=n,\\
&x_{n-1}<y_{n}^{(i)}<y_{n}^{(j)}<x_n<y_{n+1}^{(i)}, \quad &\text{ if } i<j=n+1,\\
&x_{n-1}<y_{n}^{(j)}<x_n<y_{n+1}^{(i)}, \quad &\text{ if } i=r=n<j=n+1 \text{ such that } a_r\ne 0.
\end{aligned}
\end{equation}
These observations guarantee that the first root of $P_i$ precedes the roots of $P_j$ and also the interlacing property of the roots of $P_i$ and $P_j$, when $i<r<j$ or if $i=r<j\le n+1$, $a_r\ne 0$.

Finally, we focus on the cases $r< i<j \le n+1$ or $i=r<j\le n+1$ with $a_r=0$. Then, \eqref{Jistardef} implies for these cases, that $0\notin \Ji$ and $0\notin \Jj$. Applying again Lemma \ref{l:zerosinsame} we have that inequalities appearing in \eqref{roots} and the corresponding chain of inequalities \eqref{roots_around_xi} or \eqref{roots_around_xi2} around $x_i$ and $x_j$ follow. Concerning the end of the chain of the roots, we have that
\begin{equation*}
\begin{aligned}
&x_{n-1}<y_{n}^{(j)}<y_{n}^{(i)}<x_n<y_{n+1}^{(j)}<y_{n+1}^{(i)}, &\text{ if } i<j < n,\\
&x_{n-1}< y_{n}^{(i)}<x_n<y_{n+1}^{(j)}<y_{n+1}^{(i)}, &\text{ if } i<j = n,\\
&x_{n-1}<y_{n}^{(i)}<y_{n}^{(j)}<x_n<y_{n+1}^{(i)}, \quad &\text{ if } i<n, j=n+1,\\
&x_{n-1}<y_{n}^{(j)}<x_n<y_{n+1}^{(i)}, \quad &\text{ if } i=n, j=n+1.
\end{aligned}
\end{equation*}
Moreover, in the case $1=r=i$, we know that $a_1=a_r\ne 0$ (since $(-1)^{n+1}a_1<0$), which has already been treated in the previous case ($i=r<j$ with $a_r\ne 0$). Therefore, it remains to consider the cases $(1\le) r<i$ or $1<r\le i$ with $a_r=0$. In both cases, we have  $0\notin \Ji\cup \Jj$, while $1\in \Ji\cap \Jj$ and $y_1^{(j)}<y_1^{(i)}$. Hence, the chain of interlacing roots begins with the smallest root of $P_j$ whenever  $r< i<j \le n+1$ or when $i=r<j\le n+1$ with $a_r=0$.

Consequently, we get the desired interlacing property of roots of every pair $P_i$ and $P_j$. 
\end{proof}

\begin{corollary}\label{c:order_roots_total}
Finally, we get that the roots $y_k^{(i)}$ ($k=0,\dots,n$, $i=1,\dots,n+1$) have the following order:
\begin{align}\label{order_total}
(y_0^{(r)})<~y_0^{(r-1)}&<\dots<y_0^{(1)}< x_0 <y_1^{(n+1)}<y_1^{(n)}<\dots<y_1^{(2)}< x_1 \notag
\\ & < \dots < y_n^{(n+1)}< x_n <y_{n+1}^{(n)}<\dots<y_{n+1}^{(1)},
\end{align}
where $y_0^{(r)}$ exists, if and only if $\deg P_r=n$.
\end{corollary}

\subsection{The inheritance law for derivative roots of oscillating and nearly oscillating hybrid polynomials}\label{nearly_osc}

We want to consider the derivatives of oscillating hybrid polynomials $f \in \YY_n$ with a maximal number of $n+1$ roots in $\RR$, necessarily all simple. The family of these hybrid polynomials will be denoted as $\OO$. Further, we will need to consider "nearly oscillating hybrid polynomials" of $\YY_n$ -- whose family will be denoted by $\N$-- which have $n$ real and simple roots. Note that nearly oscillating hybrid polynomials from $\N$ comprise oscillating hybrid polynomials of degree $n-1$, i.e., ${\mathcal O}_{n-1} \subset \N$, and also weighted polynomials from $\W_n$ having $n$ zeros belong to $\N$.

The first case is relatively simple: if $f:=at^n+p+c \in \YY_n$ ($a,c \in \RR$ and $p\in \Pnn$) has $n+1$ roots, then its derivative $f' \in\WW_n$ has $n$ simple roots, exactly one between any two consecutive zeros of $f$, and that's it. For the nearly oscillating hybrid polynomials, the situation is a bit different. First, we need to analyze when this situation occurs.

If the constant term $c$ is zero, then $f \in \WW_n$, too, and $f$ can have exactly $n$ simple zeros only if $a\ne 0$. If $\deg f=n-1$, yet there are exactly $n$ simple zeros of $f$, then we must have $c \ne 0$. 
Finally, there are cases when $f\in\YY_n$ can have $n$ zeros with $\deg f=n$ (i.e., $a\ne 0$), and the constant term is nonzero. Concerning the last two cases, there is a relationship between the signs of the leading coefficient $a$ and that of the constant term $c$. 

\begin{lemma}\label{l:coeff_a_c}
Let $f\in \YY_n\setminus \YY_{n-1}$.

\begin{itemize}
\item[a)] If $f$ has $n$ real roots, then $\sign (ac)\ge 0$. 
\item[b)] If $f$ has $n+1$ real roots, then $\sign (ac)<0$.
\end{itemize}
\end{lemma}

\begin{remark} The above Lemma means in particular that if $f\in \N$ is a nearly oscillating hybrid polynomial such that $c\ne 0$ and the degree of its polynomial part is $n$, (i.e., $a\ne 0$), then $\sign a=\sign c$. 
\end{remark}

\begin{proof}
Since $f\in \YY_n\setminus \YY_{n-1}$, we get that $a\ne 0$. In the case $c=0$, we obtain $f\in \WW_n$, i.e., $f$ has at most $n$ roots and $\sign(ac)=0$ holds trivially.

Now assume, that $c\ne 0$, and hence $\sign (ac)\ne 0$. Let us consider the total multiplicity of sign changes of $f$. Let $t \to -\infty$, then the leading term of the polynomial part will be dominating, i.e., $f\sim ae^{-t}t^n$. We will have $(\sign f)(-T) = (\sign a)(-1)^n$, when $-T$ is ``close to $-\infty$''. On the other end, $(\sign f)(T) = (\sign f)(\infty) = \sign c$ for large $T$. Therefore, if $(\sign a)(-1)^n=\sign c$, then $f$ has an even number of roots; similarly, if $(\sign a)(-1)^{n+1}=\sign c$, then $f$ has an odd number of roots. Consequently, if $n$ and the number of zeros of $f$ have the same parity, then  $\sign a=\sign c$. In particular, if $f$ has $n$ roots, then $\sign (ac)>0$. Otherwise, if the parities differ--for example, if $f$ has $n+1$ roots-- we obtain $\sign(ac)<0$.

\end{proof}

With these observations, let us see how the derivative zeros can be situated.

\begin{lemma}\label{l:derivative_zero}
We have the following.
\begin{itemize}
\item[1)]
If $f\in \YY_n$ is an oscillating hybrid polynomial, then there are exactly $n$ simple zeros of $f'$ strictly situated between the zeros of $f$.

\item[2)]
If $f\in \YY_n$ is nearly oscillating such that $a(f)=0$, that is, $f\in \YY_{n-1}$, then there are exactly $n-1$ simple zeros of $f'$ strictly situated between the zeros of $f$. 

\item[3)]
If $f\in \YY_n$ is nearly oscillating such that $c(f)=0$, that is, $f\in \WW_{n}$, then there are exactly $n-1$ simple zeros of $f'$ strictly situated between the zeros of $f$, and there must be one more zero of $f'$ exceeding the largest root of $f$.

\item[4)]
If $f\in \YY_n\setminus (\YY_{n-1}\cup \WW_n)$ is a nearly oscillating hybrid polynomial, with precisely $n$ simple zeros, then there are $n-1$ simple zeros of $f'$ strictly situated between them, and there must be one more zero of $f'$ exceeding the largest root of $f$.
\end{itemize}
\end{lemma}

\begin{proof}
The first and second assertions are easy to obtain.
For the oscillating hybrid polynomial, it is obvious that there are exactly $n$ simple zeros of $f'$ strictly situated between the zeros of $f$. Since $f'\in \WW_n$, then $f'$ has at most $n$ zeros, so we found each of them. Similarly, if $f\in \YY_n$, such that $a(f)=0$, then $f'\in \WW_{n-1}$. There are $n-1$ simple roots of $f'$ between the zeros of $f$, which guarantees that $f'$ cannot have any more zeros.

The third case is well known and can be found, for example, in the paper \cite{MN} by Milev and Naidenov.  They used the terminology, that the oscillating weighted polynomials\footnote{Note that with respect to $\WW_n$, our $f$ with $n$ zeros counts to be an oscillating weighted polynomial for $\WW_n$, but only a nearly oscillating polynomial in $\YY_n$, given that in the latter the dimension is one higher and the maximal possible number of the zeros is $n+1$.} of $\WW_n$ satisfy Property (P), with $\delta_0=0$ and $\delta_n=1$.

Concerning the last case, i.e., if $f$ is only a nearly oscillating hybrid polynomial with precisely $n$ simple real zeros $u_0<u_1< \ldots< u_{n-1}$, such that $ac\ne 0$, then there are $n-1$ simple zeros of $f'$ strictly between them. As for the sign of $f$ on $(u_{n-1},\infty)$, we know that $f\ne 0$ on this interval, hence it does not change sign, and $\lim_\infty f = c$, so that $\sign f = \sign c$ constant on $(u_{n-1},\infty)$. We have that $\lim_{t\to u_{n-1}+0} f'/f = +\infty$, hence $\sign f' =  \sign f = \sign c$ on the interval $(u_{n-1},u_{n-1}+\varepsilon)$ with a small enough $\varepsilon>0$, whereas $f' \sim -a t^n e^{-t}$ towards $+\infty$, so that $\sign f'(T) = -\sign a$ for very large $T$. By Lemma \ref{l:coeff_a_c}, we get that for a nearly oscillating hybrid polynomial $f \in \YY_n$ (with exactly $n$ simple real zeros) we always have $\sign (ac)\ge 0$, hence  $ac\ne 0$ guarantees that $\sign (ac)>0$, so that $a$ and $c$ are of the same sign. It means that in $(u_{n-1},\infty)$ $f'$ has different signs at the ends, and so it must have, in total, an odd number of sign changes, which under the circumstances is to be exactly one sign change at a simple zero of $f'$.
\end{proof}

\begin{theorem}[Markov-type Inheritance Theorem for Hybrid Systems]\label{th:hybrid-Markov}
Let $f, g\in \YY_n$ be oscillating, or nearly oscillating hybrid polynomials such that $c(f)=c(g)=1$ and if both are nearly oscillating, then let $a(f)\ge a(g)\ge 0$.

If the zeros $u_i$ of $f$ and $v_i$ of $g$ are interlacing in the ordering $u_0<v_0<u_1<v_1<\dots$, then the zeros $\xi_i$ and $\eta_i$ of their derivatives $f'$ and $g'$ interlace in the same manner, i.e., $\xi_1<\eta_1<\xi_2<\dots$.

The number of the zeros of $f$ and $g$ and of their derivatives depend on the type of functions $f$ and $g$. More precisely,
\begin{itemize}
\item[1)]
If $f$ and $g$ are oscillating hybrid polynomials with interlacing zeros $u_0<v_0<u_1<\dots<u_n<v_n$, then the zeros $\xi_i$ and $\eta_j$ of $f'$ and $g'$ are interlacing in the order: $\xi_1<\eta_1<\xi_2<\dots<\xi_n<\eta_n$.
\item[2)]
If $f$ is an oscillating and $g$ is a nearly oscillating hybrid polynomial such that $a(g)=0$, with interlacing zeros  $u_0<v_0<u_1<\dots<v_{n-1}<u_n$, then $f'$ has $n$ and $g'$ has $n-1$ zeros and these zeros are interlacing in the order: $\xi_1<\eta_1<\xi_2<\dots<\eta_{n-1}<\xi_n$.

\item[3)]
If $f$ is an oscillating and $g$ is a nearly oscillating hybrid polynomial such that $a(g)>0$, with interlacing zeros  $u_0<v_0<u_1<\dots<v_{n-1}<u_n$, then both derivative functions $f'$ and $g'$ have $n$ zeros and these zeros are interlacing in the order: $\xi_1<\eta_1<\xi_2<\dots<\xi_n<\eta_n$.

\item[4)]
If both $f, g \in \N$ such that $(a(f)\ge) a(g)=0$ and $a(f)>0$ with interlacing zeros $u_0<v_0<u_1<\dots<u_{n-1}<v_{n-1}$, then $f'$ has $n$ and $g'$ has $n-1$ zeros and these zeros are interlacing in the following order: $\xi_1<\eta_1<\xi_2<\dots<\eta_{n-1}<\xi_n$.

\item[5)]
If both $f, g \in \N$ such that $(a(f)\ge) a(g)>0$ with interlacing zeros $u_0<v_0<u_1<\dots<u_{n-1}<v_{n-1}$, then both derivative functions $f'$ and $g'$ have $n$ zeros and these zeros are interlacing in the following order: $\xi_1<\eta_1<\xi_2<\dots<\xi_n<\eta_n$.

\end{itemize}
\end{theorem}

\begin{remark}\label{rem:hybrid-Markov}
We note that in our investigation, it is enough to clarify the above-mentioned cases. Applying Lemma \ref{l:interlacing} and \eqref{ai_order}, in the following we collect the different types of pairs $P_i, P_j$, that correspond to the cases listed in Theorem \ref{th:hybrid-Markov} and will appear in the proof of our results. Let 
\begin{itemize}
\item[1)] $i<j< r$ (and if $a(P_r)\ne 0$, then $j=r$ is also allowed). Take $f:=(-1)^{n+1-j} P_j,$ $g:=(-1)^{n+1-i}P_i$. Then, $f,g\in \YY_{n}\setminus \YY_{n-1}$ are oscillating hybrid  polynomials falling into Case 1) above.
\item[2)] $i<j=r$, such that $a(P_r)=0$. Put $f:=(-1)^{n+1-i}P_i,$ $g:=(-1)^{n+1-j}P_j$. Then, $f\in \YY_{n}\setminus \YY_{n-1}$ is an oscillating hybrid polynomial and $g\in \YY_{n-1}\cap \N$ is a nearly oscillating hybrid polynomial (with $a(g)=0$), and they fall into Case 2) above.
\item[3)] $i<r<j$ (and if $a(P_r)\ne 0$, then $i=r$ is also allowed).  Take  $f:=(-1)^{n+1-i}P_i,$ $g:=(-1)^{n+1-j}P_j$. Then, $f\in \OO$ is an oscillating hybrid polynomial and $g\in (\YY_n\setminus \YY_{n-1}) \cap \N$ is a nearly oscillating hybrid polynomial (with  $a(g)\ne 0$). This falls into Case 3).
\item[4)] $i=r<j$ such that $a(P_r)=0$. Put $f:=(-1)^{n+1-j}P_j,$ $g:=(-1)^{n+1-i}P_i$. Then, $f\in (\YY_{n}\setminus \YY_{n-1})\cap \N$ and $g\in \YY_{n-1}\cap \N$ are nearly oscillating hybrid polynomials with $a(f)>a(g)=0$, and their pair belonging to Case 4) above.
\item[5)] $r<i<j$. We can take $f:=(-1)^{n+1-j}P_j,$ $g:=(-1)^{n+1-i}P_i)$. Then, $f,g \in (\YY_{n}\setminus \YY_{n-1})\cap \N$ are nearly oscillating hybrid polynomials, with $a(f)>a(g)>0$ and the case belongs to Case 5) above.
\end{itemize}

In each case the smallest zero $u_0$ of $f$ precedes the zeros of $g$, and their zeros are interlacing.
\end{remark}

\begin{remark}
As we have seen, the important pairs of nearly oscillating hybrid polynomials (that play a crucial role in our investigation) are studied in Theorem \ref{th:hybrid-Markov}. 

Originally, in \cite{Markov}, the Markov property was proved for (unweighted) oscillating polynomials by Markov.
It is well-known that for $f,g \in \WW_n$, i.e., when $c(f)=c(g)=0$, the Markov-type property also holds, see \cite{MN}.

However, there are pairs of nearly oscillating hybrid polynomials in $\YY_n$ for which the Markov property does not hold. 

\begin{example}
Let $n=2$ and $f:=1-10e^{-t}(t-1)$ and $g:=1+50e^{-t}(t-2)(t-4)$. Then, $f\in  \YY_{n-1} \cap \N$ and $g\in (\YY_n\setminus \YY_{n-1}) \cap \N$ and the zeros $u_i$ and $v_i$ of $f$ and $g$ are  interlacing, i.e., $u_0<v_0<u_1<v_1$, but for the zeros of derivatives we find that $\xi_1<\eta_1<\eta_2$. See the picture below:

\begin{figure}[ht]
\includegraphics[width=12cm]{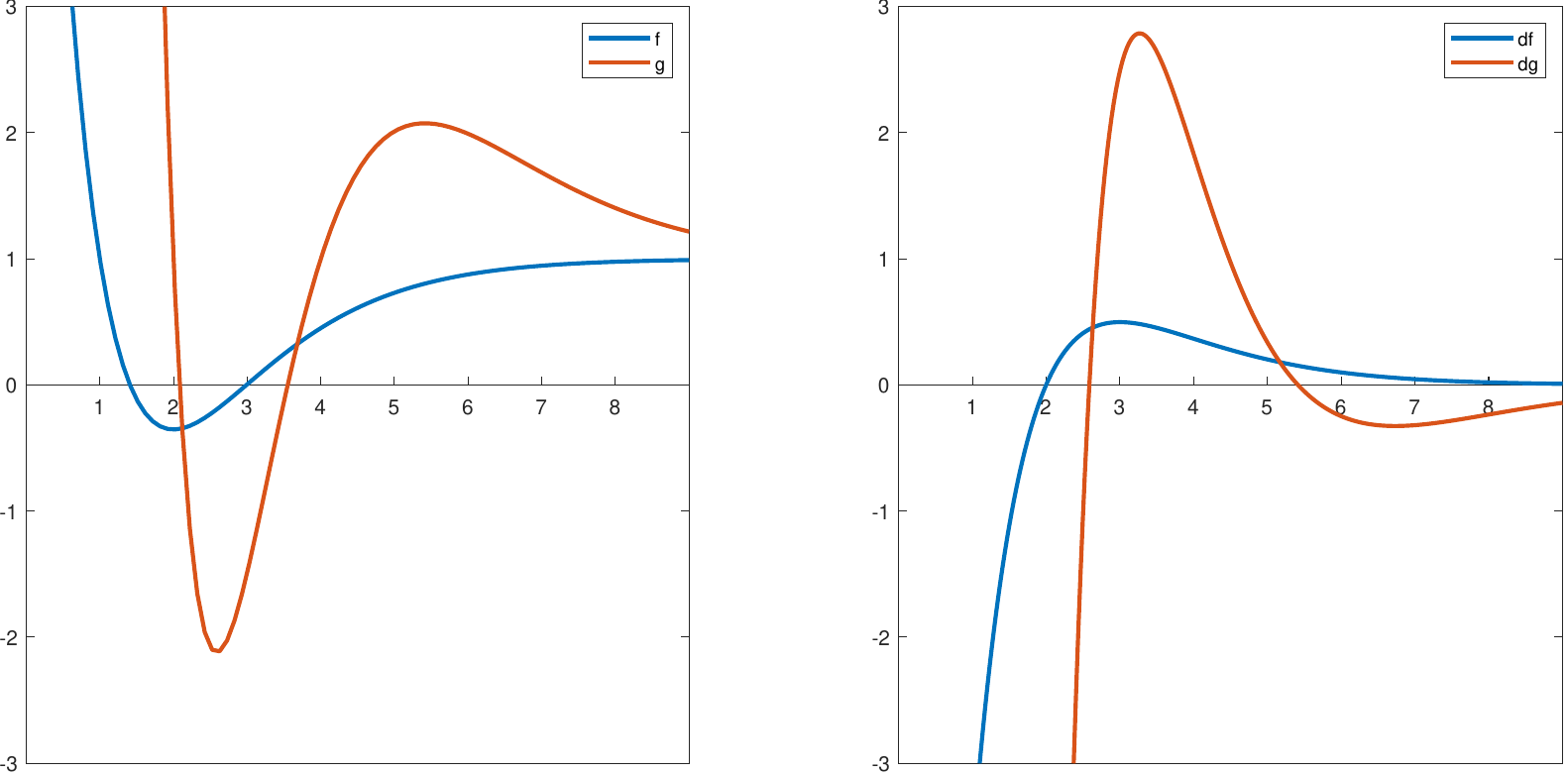}
\caption{}
\end{figure}
\end{example}

\end{remark}

In the following, we will present a proof of the Inheritance Theorem \ref{th:hybrid-Markov}, adapting an insightful method of Bojanov from \cite{Bojanov2002}. The key achievement of this method is that, unlike classical approaches for polynomials, where interpolation representations play a key role, it also works for systems where the auxiliary functions cannot be represented by interpolation. Also note that we can prove the inheritance for pairs of oscillating hybrid polynomials of our ECH system, or when at least one of them is an oscillating hybrid polynomial, but in case both are only nearly oscillating, there are cases where the statement simply fails to hold, as the above example shows. Therefore, the fact that inheritance holds in all cases needed for the main argument can be considered sheer luck. (Naturally, to clarify matters with exact proofs required some work.)

\begin{lemma}[Bojanov-type sign preservation lemma]\label{l:signpreserving} Under the above conditions --that is, in Cases 1--5) of Theorem \ref{th:hybrid-Markov}-- $\phi(x):=f'(x)g(x)-f(x)g'(x)$ has constant sign $\ve$ on $\RR$.
\end{lemma}
\begin{proof}

Fix any value $z\in \RR$ and consider the function $\psi(x):=f(x)g(z)-f(z)g(x) \in \YY_n$. Obviously, $\psi(z)=0$.

We claim that $\psi'(z)\ne 0$. If this is proved, then the Lemma follows immediately, since differentiating $\psi$ and evaluating at $z$ gives $\psi'(z)=f'(z)g(z)-f(z)g'(z)=\phi(z) \ne 0$ for any $z \in \RR$.

To prove the claim $\psi'(z)\ne 0$, first consider the case where $z=u_i$ or $z=v_j$. Then $\psi'(z)=f'(z)g(z)$ or $-f(z)g'(z)$. If this expression vanished, it would imply either that $f$ or $g$ has a double zero, contradicting our assumptions, or that $z$ is a common zero of $f$ and $g$, which is also excluded.

Observe that $g$ takes alternating signs in the intervals $(v_{i-1},v_i)$, since its zeros are simple by assumption. This also holds for $i=0$ if we set $v_{-1}=-\infty$. It follows that the points $u_i$, lying in consecutive intervals between the $v_i$, satisfy that $g(u_i)$ have alternating signs. As $f(z) \ne 0$, that means $n+1$ (or only $n$, if $f \in \N$) strict alternation of signs for $\psi(u_i)=f(z)g(u_i)$, too. These generate $n$ (or $n-1$) distinct zeros of $\psi$, say $\tau_i \in (u_{i-1},u_i)$, ($i=1,\ldots,n$ or $i=1,\ldots,n-1$).

Now, we need to separate various cases in the proof. These cases correspond to those listed in Theorem \ref{th:hybrid-Markov}.

First, consider Cases 1), 2), and 3), where $f$ is an oscillating hybrid polynomial and $g$ is either an oscillating or a nearly oscillating hybrid polynomial. Then, $f$ has $n+1$, and hence $\psi$ has $n$ distinct zeros.

We consider first the case when $C:=c(\psi)=0$. In this case, $\psi \in \WW_n$ as well. Therefore, the already found zeros $\{\tau_i~:~i=1,\ldots,n\}$ amount to the maximal number possible in $\WW_n$, so all must be simple, and there can be no additional zeros. Since $\psi(z)=0$, we must have $z=\tau_i$ for some $i$, and by simplicity of the zeros, $\psi'(\tau_i)\ne 0$. Hence $\psi'(z)\ne 0$, as required.

Second, assume now that $C \ne 0$. As before, we already know of $n$ zeros
$\tau_i$, ($i=1,\ldots,n$) with $\tau_i\in (u_{i-1},u_i)$. If $\psi'(z)$ was $0$, then $z$ is a zero of multiplicity at least two, so $z \not \in \{\tau_i~:~i=1,\ldots,n\}$ would mean an additional 2 zero multiplicity, already too many ($n+2$, whereas in $\YY_n$ only $n+1$ is possible). So, we must have $z \in \{\tau_i~:~i=1,\ldots,n\}$, i.e., $z=\tau_i$ with some $i$. But $\tau_i \in (u_{i-1},u_i)$, and $v$ takes opposite signs at the endpoints\footnote{The intervals $(u_{i-1},u_i)$ start with $(u_0,u_1)$ and end with $(u_{n-1},u_n)$, so here we don't need to extend considerations to infinite points: the endpoints, where we claimed alternation of signs, are finite nodes.} $u_{i-1}$ and $u_i$, so that the multiplicity of all the zeros in the interval, reaching at least two, must then be at least three. That is again too many: $\ge 3 +(n-1)=n+2$.

Summing up, whether we assumed $C=0$ or not, the claim that $\psi'(z) \ne 0$ is proven.

\vspace{2mm}
Now, we consider Cases 4) and 5), when both $f, g$ are nearly oscillating hybrid polynomials, with zeros $u_0<v_0<u_1<v_1<\ldots<u_{n-1}<v_{n-1}$.

The first part, involving the definition $\psi(x) := f(x) g(z) - f(z) g(x)$ and the verification that $\psi'(z) \ne 0$ implies the claim, as well as that this holds for $z = u_i$ or $z = v_j$, proceeds mutatis mutandis.

So, as before, we may assume that $z$ is a zero of neither function. Let $a:=a(f)$ and $b:=a(g)$. Then, by assumption we know that $a-b\ge 0$ and $A:=a(\psi)$ can be obtained as $A=ag(z)-bf(z)$ and exploiting that we normalized to have $c(f)=c(g)=1$, the constant term of $\psi$ becomes $C:=c(\psi)=c(f) g(z)- c(g) f(z)=g(z)-f(z)$. Observing the alternation of $\psi$ we obtain, similarly as before, now $n$ strict alternation of signs and hence $n-1$ distinct zeros of $\psi$ in $(u_0,u_{n-1})$. Consequently, if $\psi$ has only $n-1$ zeros, then each zero of $\psi$ is simple, and there is no double zero $z$ of $\psi$.

Next, let us discuss the case when $\psi$ has $n$ zeros (also counted with multiplicity). Recall that according to Lemma \ref{l:coeff_a_c}, this case occurs exactly for $AC \ge 0$. If this is the case, then in the interval $(u_0,u_{n-1})$ there can be no more zeros of $\psi$, for the strict alternation of signs between the $u_i$ would then force us to have at least $n-2+3=n+1$ zeros, which exceeds $n$, maximally possible. Now, if $z\in \RR\setminus [u_0,u_{n-1}]$, then $\psi$ has $n-1$ zeros in $(u_0,u_{n-1})$, and only $n$ altogether, so again $z$ must be simple and $\psi'(z)\ne 0$. 

Assume now that $\psi$ has $n+1$ zeros. Then, we know that $AC<0$, i.e.,
\begin{align*}
AC=ag^2(z)+bf^2(z)-(a+b)f(z)g(z)=(a-b)(g^2(z)-f(z)g(z))+b(g(z)-f(z))^2<0.
\end{align*}
Since $a\ge b\ge 0$, we must also have $f(z)g(z)>g^2(z)$, or equivalently $f(z)g(z)>0$ and $|f(z)|>|g(z)|$. It implies that besides the roots $\tau_i\in (u_{i-1},u_i)$, there exists at least one further root in $(u_{n-1},\infty)$. Indeed,
\[
(\sign \psi)(\infty)=\sign C=\sign(g(z)-f(z))=\sign (-f(z)),
\]
for $|f(z)|>|g(z)|$ forces the sign of $-f(z)$ ruling here, whereas $u_{n-1}<v_{n-1}$ entails that $\sign g(u_{n-1})=-\sign g(\infty)$, leading to
\[
(\sign \psi)(u_{n-1})=\sign(-f(z)g(u_{n-1}))=\sign(f(z)g(\infty))=\sign f(z),
\]
in view of $g(\infty)=c(g)=1$ by assumption.

Moreover, the $n$ strict alternation of signs between the $u_i$, now extended with $u_n:=\infty$, that in each of the intervals $(u_{i-1},u_i), ~ i=1,\ldots,n$, the number of zeros (counted with their multiplicities) can not be greater than one. Indeed, otherwise, the strict alternation of signs of $\psi$ at the endpoints of $(u_{i-1},u_i)$ forces a zero with multiplicity at least 3, which is impossible ($n-1+3=n+2>n+1$). Consequently, $\psi$ has only simple zeros.

Altogether, we arrived at having only simple zeros in all cases of $\psi$ having $n-1$, $n$ or $n+1$ zeros, hence the proof of the Lemma in Cases 4) and 5) is finished, too.
\end{proof}

\begin{proof}[Proof of Theorem \ref{th:hybrid-Markov}] By the above Lemma, $\phi(t):=f'(t)g(t)-f(t)g'(t)$ has constant sign $\ve$ on the whole real line $\RR$.

We need to divide the proof into several cases, similarly to the proof of Lemma \ref{l:signpreserving}. 

Firstly, we consider Cases 1) -- 3). Let $\xi_i \in (u_{i-1},u_i)$ be the roots of $f'$--there are $n$ such intervals, hence roots of $f'$. Since $f' \in \WW_n$, there can be no more, and all of them must be simple.  Furthermore, $f'$ changes its sign alternatingly on consecutive intervals $(\xi_{i-1},\xi_i)$, including the end intervals from $\xi_{0}:=-\infty$ and to $\xi_{n+1}=+\infty$. In particular, Lemma~\ref{l:signpreserving} implies that 
\begin{equation}\label{vxii} 
\phi(\xi_i)=-f(\xi_i)g'(\xi_i) 
\end{equation} 
has the same sign $\varepsilon$ for all $i$. Note that there is precisely one $u_i$ in $(\xi_{i},\xi_{i+1})$, again withstanding the infinite end intervals. So, $f(\xi_i)$ must be of alternating signs. It follows from \eqref{vxii} and Lemma \ref{l:signpreserving} that also $g'(\xi_i)$ are of alternating signs. That leads to finding a zero $\eta_i$ of $g'$ in each intervals $(\xi_i,\xi_{i+1})$ for $i=1,\ldots,n-1$. 

Consequently, when $g\in \YY_{n-1}$ and hence $g'\in \WW_{n-1}$ (that falls into Case 2), we get that
\[
\xi_1<\eta_1<\dots<\xi_{n-1}<\eta_{n-1}<\xi_n,
\]
i.e., the interlacing property for the zeros of $f'$ and $g'$ is proven.

However, in Cases 1) and 3), for $g' \in \WW_n$ there exists still another simple zero $\eta^*$ outside $(\xi_1,\xi_n)$. More precisely, by Lemma \ref{l:derivative_zero} we know, that since $g\in (\YY_n\setminus \YY_{n-1})\cap \N$, the roots $\eta_i$ $(i=1,\dots,n)$ of $g'$ are lying in the intervals $(v_{i-1},v_i)$ for $i=1,\dots, n$, with $v_n:=\infty$. Therefore, the last root $\eta^*$ of $g'$ is strictly greater, than $v_{n-1}>\xi_1$, and hence it is necessarily greater than $\xi_n$, i.e., we get that
\[
\xi_1<\eta_1<\dots<\xi_n<\eta_n.
\]
It means, that the interlacing of roots $\xi_i$ and $\eta_i$, with $\eta_n:=\eta^*$ (in the same order as those of $u_i$ and $v_i$) are proven.

In Cases 4) and 5), we know that $f\in (\YY_{n}\setminus (\YY_{n-1}\cup \WW_n))\cap \N$, i.e., $f$ has $n$ distinct roots ($u_0<\dots <u_{n-1}$). 
Then, the Part 4) of Lemma \ref{l:derivative_zero} guarantees, that $f'$ has simple roots in each interval $(u_{i},u_{i+1})$, $i=0, \dots, n-1$, with $u_{n}:=\infty$. Therefore, we get, that there exists a unique $u_i\in (\xi_{i},\xi_{i+1})$ for $i=0, \dots, n-1$, where $\xi_0:=-\infty$, implying that $f(\xi_i)$'s have alternating signs. Therefore, as in the previous cases, Lemma \ref{l:signpreserving} and \eqref{vxii} imply the alternating signs of $g'(\xi_i)$, which leads to find $n-1$ simple roots of $g'$ in the intervals $(\xi_{i},\xi_{i+1})$, $i=1,\dots, n-1$. It implies for Case 4), i.e., when $g\in \YY_{n-1}\cap \N$, that the interlacing property of the roots of $f'$ and $g'$ (in the same order as those of $u_i$ and $v_i$) is proven, i.e., we get that
\[
\xi_1<\eta_1<\dots<\xi_{n-1}<\eta_{n-1}<\xi_n.
\]

Otherwise, in Case 5), if $g\in \YY_n\setminus \YY_{n-1}$, then (similarly to Cases 1) and 3)), there exists a further zero $\eta^*:=\eta_n$ of $g'$, that is necessarily greater than $\xi_{n}$. Consequently, the interlacing property of the zeros $f'$ and $g'$  holds again with
\[
\xi_1<\eta_1<\dots<\xi_n<\eta_n.
\]

\end{proof}

\subsection{Nonsingularity}
 
In what follows, we apply the same crucial tool that appeared in the original proof of the Bernstein and Erdős Conjectures for the ``classical'' polynomial system. Namely, we study and establish the nonsingularity properties of matrices obtained by the partial derivatives $\frac{\partial m_i}{\partial x_j}(\xx)$ of the interval maxima $m_i$ on $I_i$ of the Lebesgue function for $i=1,\dots,n+1,$ $i\ne k$ and $j=1,\dots,n$.

In order to work out an explicit formula for these partial derivatives, we first apply a result of Kilgore and Cheney (see Lemma 1 in \cite{Kilgore-Cheney}). 
\begin{lemma}[Kilgore-Cheney]\label{l:K-Ch}
If an ECHS contains the constant $\bf{1}$ function, then the Lebesgue function $L(\xx,\cdot)=\sum\limits_{k=0}^{n+1} |h_k(\xx,\cdot)|$ is strictly greater than 1 on the interior of $I_i$ for all occuring $i$.
\end{lemma}

Recall, that $L(\xx,t)=1$, if $t=x_j$ for some $j\in\{0,1,\ldots,n\}$, (where $x_0=0$ and $\xx=(x_1,\ldots,x_n)$).
Note that Kilgore and Cheney investigated the interpolation process on a finite interval, but their proof works for $\II=[0,\infty)$, too.

\begin{lemma}\label{l:zi_interior}
For fixed $\xx \in S$ and $i \in \{1, \ldots, n+1\}$, there exists a
unique interior point $z_i(\xx)$ of $I_i$ where $L(\xx,t)=P_i(\xx,t)$
attains its maximum on $I_i$.
\end{lemma}

\begin{proof}
According to Lemma \ref{l:K-Ch}, the Lebesgue function $L(\xx,t)> 1$ for every $t\in \II$  that is not a node. Therefore, the function $P_i(\xx,t)\ge 1$ on $I_i$ with equality if and only if $t=x_{i-1}, x_i$. Applying Rolle's theorem we get that $(P_i)'_t$ has at least one zero in $(x_{i-1},x_i)$. We show that it is a unique root of $(P_i)'_t$ in $\intt I_i$ for every $i=1,\dots, n+1$.

If $i<r$ (or $i\le r$, if $a_r\ne 0$), then \eqref{Jistardef} implies, that the function $P_i$ has a root in each interval $I_k$, where $k\ne i$, $k=0,\dots, n+1$, implying $n+1$ roots of $P_i$. So, there are $i$ roots up to $x_{i-1}$ and $n+1-(i+1)+1$ roots that are greater, than $x_{i}$. Consequently, $(P_i)'_t$ must have zeros situated strictly between the roots of $P_i$, which altogether means at least $n$ different roots. Since $P_i\in \YY_n$, the derivative  $(P_i)'_t \in \WW_n$, and it has no more zeros. As seen, there is a root of $(P_i)'_t$ in $(x_{i-1},x_i)\subset (y_{i-1}^{(i)},y_{i+1}^{(i)})$ (where $y_{i-1}^{(i)},y_{i+1}^{(i)}$ are the roots of $P_i$ in $I_{i-1}$ and $I_{i+1}$, respectively). Thus, $z_i(\xx)$ is a unique zero of $(P_i)'_t$ in $\intt I_i$.

We note in passing that this zero-counting also means that $(P_1)'_t$ cannot have more than one zero between the zeros of $P_1$ found in $(-\infty, x_0)$ on the one hand and in $I_2$ on the other hand. That is, the zero in $I_1$ of $(P_1)'_t$ is the only zero in $J_1=(-\infty,y_2^{(1)})$, and hence its first zero, too.

If $i=r$ and $a_r=0$, then $P_r$ has $n$ roots, one in each interval $I_k=[x_{k-1},x_k]$, where $k\ne r$, $k=1,\dots, n+1$ 
and $(P_r)'_t$ has $n-1$ roots between the roots of the ones of $P_r$. Therefore, similarly to the previous case, there is a unique root of $(P_r)'_t$ in $(x_{r-1},x_r)$.

If $r<i<n+1$, then $P_i$ is a nearly oscillating hybrid polynomial, that has $n$ roots, one in each interval $I_k=[x_{k-1},x_k]$, where $k\ne i$, $k=1,\dots, n+1$. Concerning the derivative function $(P_i)'_t$, we know that it has also $n$ roots, such that $n-1$ roots are between the roots of the function $P_i$ and there is an additional zero beyond the largest root of $P_i$. It implies again, that $(P_i)'_t$ has a unique root $z_i(\xx)$ in $\intt I_i$, as desired.

Finally, a similar argument applies to $P_{n+1}$. There is one root  in each interval $I_k=(x_{k-1},x_k)$, where $k=1,\dots, n$. Consequently, its derivative $(P_{n+1})'_t$ has $n-1$ roots before $x_n$, and since $P_{n+1}(\xx,x_n)=\lim_{t\to \infty} P_{n+1}(\xx,t)=1$, there is an extra, unique root of $(P_{n+1})'_t$ in $(x_n,\infty)$, which implies, that $z_{n+1}(\xx)$ is also well-defined.

In particular, $z_{n+1}(\xx)$ is the largest zero of $(P_{n+1})'_t$.
\end{proof}

Based on the ideas of \cite{Kilgore-Cheney}, we can prove the following identity for the partial derivatives $\frac{\partial m_i}{\partial x_j} (\xx)$.

\begin{lemma}\label{l:Kilgore_Cheney}
Consider the interpolation operator associated with an ECH system of functions on the interval $\II=[0,\infty)$ and let $\xx\in S$. Then,
\begin{itemize}
\item[1)]
\[
\frac{\partial  P_i}{\partial x_j} (\xx,t)=-h_j(\xx,t)(P_i)'_t(\xx,x_j), \quad (t\in \RR).
\]
\item[2)]
If $\bf{1}$ belongs to the ECH system, then 
\[
\frac{\partial m_i}{\partial x_j}(\xx)=\frac{\partial P_i}{\partial x_j} (\xx,z_i(\xx))\bigg(=-h_j(\xx,z_i(\xx))(P_i)'_t(\xx,x_j)\bigg).
\]
\end{itemize}
\end{lemma}

\begin{proof}
Concerning Part 1), let $\xx\in S$ and $\xx'\in S$ be given by $x'_k=x_k$ for each $k\in \{1,\ldots,n\}\backslash\{j\}$ and $x_j'\in (x_{j-1},x_j)$. According to Lemma 2 of Kilgore and Cheney in \cite{Kilgore-Cheney}, we have
\begin{equation}\label{K-Ch_diff}
\begin{aligned}
P_j(\xx,t)-P_j(\xx',t)&=(P_j(\xx,x_j')-1)h_j(\xx',t)\\
P_{j-1}(\xx,t)-P_{j-1}(\xx',t)&=(P_{j-1}(\xx,x_j')-1)h_j(\xx',t).
\end{aligned}
\end{equation}
We now extend this result to an arbitrary $P_i$ not necessarily equal to $P_j$ or $P_{j-1}$. Namely, we are going to prove that
\begin{equation}\label{K-Ch_diff_j}
P_i(\xx,t)-P_i(\xx',t)=(P_i(\xx,x_j')-\varepsilon_{j,i})h_j(\xx',t).
\end{equation}

Observe that the function $P_i(\xx,t)-P_i(\xx',t)$ belongs to the ECH system and vanishes at $n$ nodes ($x_k$, $k=0,\ldots,n$, $k\ne j$). Since $h_j(\xx',t)$ has zeros at the same points, the ECHS property guarantees that there exists a constant $c$ such that
\[
P_i(\xx,t)-P_i(\xx',t)=c\cdot h_j(\xx',t).
\]
Moreover, the corresponding constant factor is given by
\[
c=\restr{\frac{P_i(\xx,t)-P_i(\xx',t)}{h_j(\xx',t)}}{t=x'_j}=P_i(\xx,x_j')-\varepsilon_{j,i},
\]
proving \eqref{K-Ch_diff_j}.
Consequently, we obtain Part 1) by the following calculation:
\begin{equation*}
\begin{aligned}
-\frac{\partial P_i}{\partial x_j}(\xx,t)&:=\lim\limits_{x_j'\to x_j} \frac{P_i(\xx,t)-P_i(\xx',t)}{x_j'-x_j}=\lim\limits_{x_j'\to x_j} \frac{(P_i(\xx,x_j')-\varepsilon_{j,i})h_j(\xx',t)}{x_j'-x_j}\\&
=\lim\limits_{x_j'\to x_j} \frac{P_i(\xx,x_j')-P_i(\xx,x_j)}{x_j'-x_j}\cdot h_j(\xx,t) =(P_i)'_t(\xx,x_j)\cdot h_j(\xx,t).
\end{aligned}
\end{equation*}

For the first equation in Part 2), we use that $z_i(\xx)\in \intt I_i$. This is guaranteed by Lemma \ref{l:zi_interior} for every $i=1,\dots,n+1$, provided that $\bf{1}$ belongs to the ECH system we work with.

Since $z_i(\xx)\in \intt I_i$, we get that $P_i'(\xx,z_i(\xx))=0$. By the Implicit Function Theorem $z_i$ is continuously differentiable, and its derivative with respect to $x_j$ is given by the formula
\begin{equation*}\label{zi_der}
\frac{\partial z_i}{\partial x_j}(\xx)= -\frac{1}{(P_i)''_{tt}(\xx,z_i(\xx))} \frac{\partial (P_i)'_t}{\partial x_j}(\xx,z_i(\xx)).
\end{equation*}

Hence, we obtain that
\[
\frac{\partial m_i}{\partial x_j}(\xx)=
\frac{\partial P_i}{\partial x_j}(\xx,z_i(\xx)) + (P_i)'_t(\xx,z_i(\xx))\frac{\partial z_i}{\partial x_j}(\xx)= \frac{\partial P_i}{\partial x_j}(\xx,z_i(\xx)).
\]

\end{proof}

Applying the previous Lemma \ref{l:Kilgore_Cheney}, and the definition \eqref{eq_hkt} of $h_j$, we obtain, by setting $z_i:=z_i(\xx)$, that
\begin{align*}
\frac{\partial m_i}{\partial x_j}(\xx)=-h_j(\xx,z_i)(P_i)_t'(x_j)
=\frac{e^{-z_i} \prod_{l=0}^n(z_i-x_l) ~ (P_i)_t'(x_j)}{e^{-x_j}\prod_{l=0,~l\ne j}^n (x_j-x_l) ~ (x_j-z_i)}.
\end{align*}

Therefore the singularity properties of the matrices formed by these partial derivatives (with leaving out one index $k$, i.e., taking $i=1,\ldots,n+1$, $i\ne k$) are equivalent to those of the derived matrices when we divide the $i$th row by $e^{-z_i} \prod_{l=0}^n(z_i-x_l)$, and multiply the $j$th column by $e^{-x_j}\prod_{l=0,~l\ne j}^n (x_j-x_l)$. This transformation was suggested by D. Braess to T. Kilgore and is duly recorded in the papers of Kilgore and Cheney \cite{Kilgore-Cheney}, Kilgore \cite{K-AMH}, and of de Boor-Pinkus \cite{CBoorPinkus}. So, equivalently to the original aim of determining the singularity properties of the matrices
$$A_k:=\left[\frac{\partial m_i}{\partial x_j}\right]_{i=1, i\ne k, j=1}^{n+1, n},$$ it suffices to analyze the singularity of the submatrices, obtained by dropping the $k$th row from
\begin{equation*}
Q:=[q_i(x_j)]_{i=1, j=1}^{n+1,n}, \quad \text{where} \quad q_i(t):=\frac{(P_i)_t'(t)}{t-z_i} \quad(i=1,\ldots,n+1);
\end{equation*}
that is,
\begin{equation}\label{Qmatrix}
\begin{aligned}
Q_k:=
\begin{bmatrix}
\frac{(P_1)'_t(x_1)}{x_1-z_1} & \frac{(P_1)'_t(x_2)}{x_2-z_1}& \dots & \frac{(P_1)'_t(x_n)}{x_n-z_1}\\[2mm] 
\vdots & \vdots & \ddots & \vdots\\[2mm]
\frac{(P_{k-1})'_t(x_1)}{x_1-z_{k-1}} & \frac{(P_{k-1})'_t(x_2)}{x_2-z_{k-1}}& \dots &
\frac{(P_{k-1})'_t(x_n)}{x_n-z_{k-1}}\\[2mm]
\frac{(P_{k+1})'_t(x_1)}{x_1-z_{k+1}} & \frac{(P_{k+1})'_t(x_2)}{x_2-z_{k+1}}& \dots &\frac{(P_{k+1})'_t(x_n)}{x_n-z_{k+1}}\\[2mm]
\vdots & \vdots & \ddots & \vdots \\[2mm]
\frac{(P_{n+1})'_t(x_1)}{x_1-z_{n+1}} & \frac{(P_{n+1})'_t(x_2)}{x_2-z_{n+1}} & \dots & \frac{(P_{n+1})'_t(x_n)}{x_n-z_{n+1}}\\
\end{bmatrix}.
\end{aligned}
\end{equation}
Note that the functions $q_i$ all belong to $\WW_{n-1}$, which forms a rank $n-1$ ECHS.

\vspace{1mm}
Our next goal is to apply Proposition 1 in \cite{Kilgore1985}, stated below for completeness.

\begin{lemma}[Kilgore]\label{l:Kilgore_orig}
Let $z_1, \dots, z_{n+1}$ be real numbers with $z_1<z_2<\dots< z_{n+1}$ and $q_1,\dots, q_{n+1}$ be functions from a rank $n-1$ ECHS. Assume that $q_1,\dots, q_{n+1}$ are such that
\begin{itemize}
\item[1)] $q_i(z_j)\ne 0$ for every $i, j=1,\dots,n+1$,
\item[2)] $q_i$ has a unique, simple root in $[z_{j},z_{j+1}]$ for all  $j\ne i-1,i$,
\item[3)] $q_i$ does not have any roots in $[z_{i-1},z_{i+1}]$.
\end{itemize}
Then, for any $k\in \{1,\dots, n+1\}$, the set $\{q_1,\dots, q_{n+1}\}\setminus \{q_k\}$ of functions is linearly independent.
\end{lemma}

\begin{remark}\label{r:places_of_zi_roots}
Regarding the complete set of zeros of the functions $q_i$ $(i=1,\dots,n+1)$, apart from those taken into account in the interval $[z_1,z_{n+1}]$ in the previous Lemma, each $q_i$ has an extra unique, simple root in $\RR\setminus [z_1,z_{n+1}]$, except the case when $i=r$ and $a_r=0$.

To see this, we can refer to the proof of Lemma \ref{l:zi_interior}, where it was established that $z_1$ is the first zero of $(P_1)'_t$. Hence, from the ordering of the respective zero sequences, we find that there are roots of $(P_i)'_t$ before $z_1$ precisely when $i\le r$ (or, in case $a_r=0$, for $i<r$), yielding altogether $r-1$ (or $r-2$) zeros; and $z_{n+1}$ is the last zero of $(P_{n+1})'_t$, hence there are roots of $(P_i)'_t$ exceeding $z_{n+1}$ precisely for $i>r$, altogether $n+1-r$ roots.
\end{remark}

\begin{lemma}\label{l:Kilgore_assump}
The functions $q_i(t):=\frac{(P_i)_t'(t)}{t-z_i}$ $(i=1,\dots,n+1)$ satisfy the properties  1) -- 3) listed in Lemma \ref{l:Kilgore_orig}.
\end{lemma}

In the proof of Lemma \ref{l:Kilgore_assump}, we need an auxiliary result about pairwise strictly interlacing sequences.  This was already used in \cite{Kilgore} and \cite{CBoorPinkus}, and ever since, but we find it necessary to give a clean formulation and a precise proof for it.

So let $\mathcal{A}'=\{ a_1'<\dots<a_i'<\dots<a_n'\}$ and $\mathcal{A}''= \{ a_1''<\dots<a_i''<\dots<a_n''\}$ be two sequences; in the following we will consider only increasing sequences, and, somewhat loosely, we identify them with the sets of their elements, too.

We say that $\A'$ precedes $\mathcal{A}''$, and we write $\mathcal{A}'\prec \mathcal{A}''$, if $a_i'<a_i''$ for every $i=1,\dots,n$ and $a_i''<a_{i+1}'$ for every $i=1,\dots, n-1$.
If either $\mathcal{A}'\prec \mathcal{A}''$ or $\mathcal{A}''\prec \mathcal{A}'$, then we say that $\A', \A''$ are strictly interlacing.

Now if we choose two elements $\alpha, \beta \in \mathcal{A}' $, then we write $\alpha \llcurly \beta$ to mean that $\alpha<\beta$ and $\A'$ does not contain any other element in $(\alpha,\beta)$. (That means $\alpha=a_i'$ and $\beta=a_{i+1}'$ for some $i=1,\ldots,n-1$.)

Further, if  $\alpha \in \mathcal{A}'$ and $\beta \in \mathcal{A}''$ for two strictly interlacing sequences $\A'$ and $\A''$, then $\alpha \llcurly \beta$ will mean that $\alpha<\beta$ and neither $\mathcal{A}'$, nor $\mathcal{A}''$ has any element in $(\alpha,\beta)$.

\begin{lemma}\label{l:interlacing_seq}
If $\mathcal{A}^{(j)}=\{ a_1^{(j)}<\dots<a_i^{(j)}<\dots<a_n^{(j)}\}$ are pairwise strictly interlacing sequences for $j=1,\dots, \nu$, such that $\mathcal{A}^{(j)}\prec \mathcal{A}^{(k)}$ whenever $j< k$, then

\begin{itemize}
\item[1)] the elements of the set $\mathcal{A}:=\bigcup_{j=1}^\nu \A^{(j)} =\{a_i^{(j)}: i=1,\dots, n, j=1,\dots, \nu \}$ are ordered lexicographically, i.e., $a_i^{(j)}\le a_l^{(k)}$ if and only if $i<l$ or $i=l$ and $j\le k$.

\item[2)] for any two elements $a_i^{(j)}, a_l^{(k)}$ of $\A$, such that $a_i^{(j)}\le a_l^{(k)}$ we have 
\begin{equation}\label{cardinality}
\#\left(\A\cap [a_i^{(j)}, a_l^{(k)}]\right)=k-j+1+(l-i)\nu.
\end{equation}

\item[3)] if an interval $[\alpha,\beta]$ contains exactly $\nu$ elements of $\A$, then these $\nu$ elements contain exactly one from each of the $\A^{(j)}$, $j=1,\ldots,\nu$.
\end{itemize}
\end{lemma}

\begin{proof}
We can visualize $\A$ as a matrix
\begin{equation}\label{matrix_A}
[a_i^{(j)}]_{i=1, j=1}^{n, \nu}= \begin{bmatrix}
a_1^{(1)} & \dots & a_1^{(\nu)}\\
a_2^{(1)} & \dots & a_2^{(\nu)}\\
\vdots & \ddots & \vdots\\
a_n^{(1)} & \dots & a_n^{(\nu)}
\end{bmatrix}.
\end{equation}

Within each row of $\A$, the ordering is determined by the upper index; $a_i^{(j)}<a_i^{(k)}$ if and only if $j<k$, because $a_i^{(j)}<a_i^{(k)}$ is just the first requirement of the definition of $\A^{(j)}\prec \A^{(k)}$.

Between elements of two consecutive rows we have $a_i^{(j)}<a_{i+1}^{(k)}$, always. If $j=k$, then this is just the monotonicity assumption on $\A^{(j)}$. If $\A^{(j)}\prec \A^{(k)}$, then the same monotonicity of $\A^{(j)}$ gives $a_i^{(j)}<a_{i+1}^{(j)}$, and then by $\A^{(j)}\prec \A^{(k)}$ we also have $a_{i+1}^{(j)} < a_{i+1}^{(k)}$, giving the assertion. Finally, if $\A^{(k)}\prec \A^{(j)}$,
then $a_i^{(j)}<a_{i+1}^{(k)}$ because of the second requirement in the definition of $\A^{(k)}\prec \A^{(j)}$.

Thus, the elements $a_i^{(j)}$ of $\A$ are ordered lexicographically by their indices $(i,j)$, proving (1).

For Part 2), we need the number of indices $(p,q)$ lying between $(i,j)$ and $(l,k)$ in the lexicographical ordering. The reader will have no difficulty in seeing that the number of indices between $(i,j)$ and $(l,k)$ is just as stated.

Finally, Part 3) follows easily from the observation that if two elements from the same $\A^{(j)}$ are in $[\alpha,\beta]$, then by interlacing all the other $\A^{(i)}$ must also have an element between these two, altogether adding up to $2+(\nu-1)>\nu$ elements of $\A$.
\end{proof}

Recall that the root sequences of the functions $P_i$ are pairwise strictly interlacing; more precisely, they follow each other as stated in Corollary \ref{c:order_roots_total}. Note that root sequences of our $P_i$ did not always contain $n+1$ elements, but could have only $n$, so that we needed the ``nearly oscillating cases'' of the respective Markov-type lemma.

Our next step is to apply the Markov-type inheritance lemma, i.e., Theorem \ref{th:hybrid-Markov}, worked out above. This furnishes that the root sequences $W^{(i)}$ of the derivatives $(P_i)'_t$ are strictly interlacing pairwise. However, from pairwise we need to obtain the global ordering, as in \eqref{order_total}, which requires additional consideration, since the Markov-type result concerns only pairwise precedences. The need to address the arising zero-counting argumentation across various intervals made it convenient to state Lemma \ref{l:interlacing_seq} in its precise form.

From the findings about how the roots of the hybrid polynomials $P_i$ follow each other, we have here from Theorem \ref{th:hybrid-Markov} that $W^{(i)}\prec W^{(j)}$ in the order described by $r \prec r-1 \prec \dots \prec 1 \prec n+1 \prec n \prec \dots \prec r+1$. In the case when $a_r=0$, we modify the definition of $W^{(r)}$ by adding an extra element to this set, such that the new element is smaller than the smallest root of $(P_{r-1})'_t$. This way we get that $W^{(i)}$ contains the same number of elements -- that is, $n$ -- for each $i=1,\ldots,n+1$, and the pairwise precedence orderings remain the same as it can be read from the ordering of the roots of the hybrid polynomials $P_i$.

We are now in a position to state the result that we have been aiming at through these technical preparations.

\begin{corollary}\label{c:zi_cyclically}
If $W:=\bigcup\limits_{i=1}^{n+1} W^{(i)}$, then $r_i:=\#\left(W\cap [z_i,z_{i+1}]\right)=n+1$ for all $i=1,\ldots,n$.
\end{corollary}

\begin{proof}
Recall that the sets $W^{(i)}$ are pairwise strictly interlacing, such that $W^{(r)}\prec W^{(r-1)}\prec\dots \prec W^{(1)}\prec W^{(n+1)}\dots \prec W^{(r+1)}$.

By Remark \ref{r:places_of_zi_roots}, we know the following cardinalities:
\begin{equation}\label{r0rnplusone}
r_0 := \#\left(W\cap (-\infty,z_{1}]\right)=r, \quad r_{n+1}:=\#\left(W\cap [z_{n+1}, \infty)\right)=n-r+1.
\end{equation}

Now we compute the total number of points in $W$ two ways. As each $W^{(i)}$ has $n$ points, all different, and there are $n+1$ sequences $W^{(i)}$, this is trivially $n(n+1)$. On the other hand, we can sum up the number of points in the subintervals of $\RR$, defined by the strictly increasing sequence of the $z_i$. Note that $z_i\in \intt I_i$ entails that the sequence is strictly increasing. Also, by definition $z_i$ is the maximum place of $P_i$ in $I_i$, so it is a root of $(P_i)'_t$, i.e., $z_i\in W^{(i)}$, always.
Therefore, we get that
\begin{equation}\label{Wfirstpart}
n(n+1)= \# W=\sum_{i=0}^{n+1} r_i-(n+1),
\end{equation}
the $-(n+1)$ occurring in the end because in the sum $\sum_{i=0}^{n+1} r_i$, the endpoints $z_i$ of the intervals -- all belonging to $W$ -- were counted twice.

For a smoother application of Lemma \ref{l:interlacing_seq}, we renumber the sequences $W^{(i)}$ to form a system of sequences ordered in increasing order of indices. So let the new sets be $\A^{(j)}$, $j=1,\ldots,n+1$, where $\A^{(j)}=W^{(i)}$ with the correspondence between the indices $i$ and $j$ given by
\begin{align*}
i:=i(j)&:=r+1-j + \left\lceil \frac{j-r}{n+1} \right\rceil (n+1),
\\
j:=j(i)&:=r+1-i + \left\lceil \frac{i-r}{n+1} \right\rceil (n+1).
\end{align*}
The indices are defined so that now $\A^{(1)}=W^{(r)} \prec \A^{(2)}=W^{(r-1)}\prec\dots \prec \A^{(n+1)}=W^{(r+1)}$, whereas $W=\A:=\bigcup_{j=1}^{n+1} \A^{(j)}$. Note that $z_i \in W^{(i)} = \A^{(j(i))}$ and $z_{i+1} \in W^{(i+1)} =\A^{(j(i+1))}$ are either of the form $a^{(j)}\in \A^{(j)}$ and $a^{(j-1)} \in \A^{(j-1)}$, or $a^{(1)} \in \A^{(1)}$ and $a^{(n+1)} \in \A^{(n+1)}$. So, in \eqref{cardinality} we either have $k=j-1$, or if $j=1$, then $k=n+1$. In both cases, Part 2) of  Lemma \ref{l:interlacing_seq} provides $r_i=\# \left(\A \cap [z_i,z_{i+1}] \right) \equiv 0 \mod \nu$ where $\nu=n+1$. Note that $r_i \ge 2$, so $r_i \ge n+1$. Substituting this estimate and \eqref{r0rnplusone} into \eqref{Wfirstpart} we find
\[
n(n+1)= \# W \ge r +\sum_{i=1}^n r_i + (n-r+1) - (n+1) = \sum_{i=1}^n r_i \ge n(n+1),
\]
so that all inequalities must be equalities, and $r_i=n+1$  for all $i=1,\ldots,n$.

We note that the assertion remains unchanged independently of $a_r=0$ or not, because the additional extra point, put in $W^{(r)}$ in case $a_r=0$, lies below all points of $W$, in particular, below $z_1 \in W^{(1)}\subset W$, hence the numbers $r_i$, $i=1,\dots,n$, do never take them into account.
\end{proof}

\begin{proof}[Proof of Lemma \ref{l:Kilgore_assump}]
Firstly, we know that the derivative functions $(P_i)'_t$ ($i=1,\dots, n+1$) have simple and pairwise strictly interlacing, hence distinct zeros. Therefore, the definition of $q_i=\frac{(P_i)'_t}{t-z_i}$ guarantees that $q_i(z_j)\ne 0$ for all $i,j =1,\dots,n+1$. This yields Property 1).

As for Property 2), Corollary \ref{c:zi_cyclically} gives that $r_j:=\#(W\cap[z_j,z_{j+1}])=n+1$.
Thus, in view of Part 3) of Lemma \ref{l:interlacing_seq}, $[z_j,z_{j+1}]$ contains a simple, unique element of each $W^{(i)}$, i.e., a simple, unique root of $(P_i)'_t$, $i=1,\dots,n+1$. Now, if $i \ne j, j+1$, then this root is also a root of $q_i$, since $q_i$ has the same roots except for $z_i$, which lies outside of $[z_j,z_{j+1}]$ in this case. We therefore have Property 2).

On the other hand, let $i=j$ or $i=j+1$, (i.e., $j=i-1$ or $j=i$). The only root of $(P_i)'_t$ in the intervals $[z_{i-1},z_i]$ resp. $[z_i,z_{i+1}]$ is exactly $z_i$, which is cancelled in $q_i$, leaving no root of $q_i$ in $[z_{i-1},z_{i+1}]$. That was Property 3).

Therefore, the functions $q_i$ ($i=1,\dots,n+1$) satisfy all three desired properties.
\end{proof}
Applying Lemma \ref{l:Kilgore_orig} (Proposition 1 in \cite{Kilgore1985}), the previous Lemma implies that the system $\{q_1,\dots,q_{n+1}\}\setminus\{q_k\}$ is linearly independent.

\begin{lemma}\label{lem:keylemma}
Assume that $\sum_{i=1}^{n+1} \ai q_i(x_j)=0$ for all $j=1,\ldots,n$, where $\al_k=0$ for a fixed $1\le k\le n+1$. Then we get that $\al\equiv 0$.
\end{lemma}

\begin{proof}
We already know that $q_1,\dots, q_{n+1}$ belong to the rank $n-1$ ECH system $\WW_{n-1}$. Since $\sum_{i=1, i\ne k}^{n+1} \ai q_i(x_j)=0$ for all $j=1,\ldots,n$, this means that $f:=\sum_{i=1, i\ne k}^{n+1} \ai q_i$ has $n$ distinct zeros, which implies that $f \equiv 0$. However, taking into account Lemma \ref{l:Kilgore_assump}, Kilgore's Lemma \ref{l:Kilgore_orig} yields that $\{q_1,\dots,q_{n+1}\}\setminus\{q_k\}$ is a linearly independent system of functions. Thus, $\al\equiv 0$ follows.
\end{proof}

\begin{corollary}\label{c:nonsingularity} The matrices $Q_k$ in \eqref{Qmatrix} are nonsingular for all $k=1,\ldots,n+1$.
\end{corollary}
\begin{proof} Assume that a linear combination, with coefficients $\alpha_i$, of the row vectors of the matrix $Q_k$ equals the zero vector.
That means $\sum_{i=1,i\ne k}^{n+1} \ai q_i(x_j)=0$ for all $j=1,\ldots,n$. By Lemma \ref{lem:keylemma}, it follows that all $\alpha_i=0$. Hence, the row vectors are linearly independent, and the matrix is nonsingular.
\end{proof}

\subsection{Properness of $\Gamma(\xx)=(m_2(\xx)-m_1(\xx), \dots, m_{n+1}(\xx)-m_{n}(\xx))$}

We recall that a continuous map $f: X\to Y$ between topological spaces is called proper if $f^{-1}(Q)(\subset X)$ is compact for any compact subset $Q$ of $Y$.

\begin{theorem}\label{t:homeom_const}
The mapping $\Gamma\colon S\to \mathbb{R}^{n}$ defined by
\begin{equation}\label{gamma}
\Gamma(\xx)=(m_2(\xx)-m_1(\xx), \dots, m_{n+1}(\xx)-m_{n}(\xx))
\end{equation}
is proper.
\end{theorem}

\begin{proof} We know, that the mapping $\xx\to m_i(\xx)$ is continuous for every $i=1,\dots,n+1$; hence the continuity of $\Gamma$ follows immediately.

Let $Q$ be a compact subset of $\RR^{n}$ and let $W:=\Gamma^{-1}(Q)$ denote its preimage under $\Gamma$. By continuity, $W$ is a (relatively) closed subset of $S$.

In the first step, we show that the boundedness of $Q$ implies that $W$ is also bounded\footnote{The whole proof is a mere technicality with the classic paper \cite{CBoorPinkus} essentially having it all. However, in that case, there was no boundedness question, given that the base interval $\II$ was taken finite, so that formally we need a proof here.}. For definiteness, assume that all coordinates of points in $Q$ are bounded by some constant $C$. If $\xx \in S$ belongs to $W$, and is thus mapped to a point of $Q$, then for all applicable indices $i$ we have $|m_i(\xx)-m_{i-1}(\xx)|\le C$, hence $\max\limits_{i=1,\ldots,n} |m_i(\xx) - m_{n+1}(\xx)| \le nC$. Since $m_{n+1}(\xx) \ge 1$, it follows that for all $i=1,\ldots,n$ we must have
\[
\frac{m_i(\xx)}{m_{n+1}(\xx)} \le nC+1.
\]
Fix a large number $q>1$ and suppose that $\xx \in S$ satisfies $x_n>q$. We will show that, for a sufficiently large choice of $q$ this point cannot belong to $W$. For sure, there exists a point $s \in [0,1]$ such that $|s-x_i|\ge 1/(2n+1)$ for all indices $i$. This point $s \in [0,1]$ belongs to some of the $I_i(\xx)$, but not to $I_{n+1}(\xx)$, for $s \le 1 <q <x_n$. Take the index $i$ with $s \in I_i(\xx)$ and let $j$ be any index with $0\le j \le n$. For $t\ge x_n$, we have
\begin{align*}
\left|\frac{h_j(\xx,t)}{h_j(\xx,s)}\right|=\frac{e^{-t}}{e^{-s}} \prod\limits_{l=0,l\ne j}^n \left|\frac{t-x_l}{s-x_l}\right| \le \frac{e^{-t}}{\min_{u\in [0,1]} e^{-u}} \frac{t^n}{1/\left(2n+1\right)^{n}}
= e^{1-t} (2n+1)^n t^n.
\end{align*}

Recall that $t\ge x_n>q$, so if $q$ is chosen large enough, then by $e^{-t}t^n \to 0$ we are led to
\begin{equation}\label{hj_estimation}
\left|\frac{h_j(\xx,t)}{h_j(\xx,s)}\right| <\ve, \qquad \text{ for all }  j=0,\ldots,n \qquad \text{ and } \qquad t\in I_{n+1}(\xx).
\end{equation}
Therefore with $t=z_{n+1}\ge x_n$, we get
\begin{align*}
m_i(\xx) \ge P_i(\xx,s)=\sum\limits_{j=0}^n |h_j(\xx,s)|+|h_{n+1}(\xx,s)|\ge \frac{1}{\varepsilon}\sum\limits_{j=0}^n |h_j(\xx,z_{n+1})|,
\end{align*}
and similarly, \eqref{hj_estimation} also holds for $t=x_n>q$, i.e.,
\begin{align*}
m_i(\xx) \ge P_i(\xx,s)=\sum\limits_{j=0}^n |h_j(\xx,s)|+|h_{n+1}(\xx,s)|\ge \frac{1}{\varepsilon}\sum\limits_{j=0}^n |h_j(\xx,x_n)|=\frac{1}{\varepsilon}.
\end{align*}
These inequalities guarantee that 
\begin{align*}
m_i(\xx) & \ge \frac{1}{3\varepsilon}\left(2\sum\limits_{j=0}^n |h_j(\xx,z_{n+1})|+1\right) \ge \frac{1}{3\varepsilon}\left(\sum\limits_{j=0}^n |h_j(\xx,z_{n+1})|+\left|1-\sum\limits_{j=0}^n h_j(\xx,z_{n+1})\right|\right)\\
&=\frac{1}{3\varepsilon}\left(\sum\limits_{j=0}^n |h_j(\xx,z_{n+1})|+\left|h_{n+1}(\xx,z_{n+1})\right|\right)=\frac{1}{3\varepsilon}m_{n+1}(\xx).
\end{align*}
Clearly, if $\ve <\frac{1}{3(nC+1)}$, then this means $\xx \notin W$.

\medskip
So $W$ is a bounded and relatively closed set; in particular, there is a constant $c$ such that $\|\xx\|\le c$ for all points $\xx \in W$. It remains to see that $W$ is closed in $\RR^n$, too. In other words, we have to show that a point $\xx \in \partial S$ with $\|\xx\|\le c$ cannot be a limit point\footnote{If we have this, then closing $W$ in $\RR^n$ results in $\overline{W} \subset \oS$, but $\partial S \cap \overline{W}=\emptyset$ so that $\overline{W} \subset S$, and $W$ being relatively closed, $W=\overline{W} \cap S = \overline{W}$ proves that $W$ is closed even in $\RR^n$.}  of $W$.

Now  let $\xx \in \DS$ such that $x_{i-1}=x_i$, say. Also let $\de>0$ be small, and assume that $\yy \in S$ is $\de$-close to $\xx$, i.e., $\|\xx-\yy\|<\de$. We take $z_i \in \intt I_i(\yy)$ the maximum point of $P_i(\yy,\cdot)$ on $I_i$, and take any other point $s \in [0,1]$ with the property that it is at least of $1/(2n+1)$ distance apart from any other point of the node system. With the unique index $k$ with $s \in \intt I_k$ we obviously have
\[
m_k(\yy) \ge P_k(\yy,s) = \sum_{j=0}^n |h_j(\yy,s)|+|h_{n+1}(\yy,s)|.
\]

Considering now the ratio of $h_j(\yy,t)$ and $h_j(\yy,s)$ for any $t\in  I_i$ and $0\le j\le n$, we are led to
\begin{align*}
\frac{|h_j(\yy,t)|}{|h_j(\yy,s)|}  = \frac{e^{-t}}{e^{-s}} \prod_{l=0, l \ne j}^n \left| \frac{t-y_l}{s-y_l}\right| 
\le \begin{cases}\frac{\max_{u\in [0,c+\de]}e^{-u}}{\min_{u\in [0,1]} e^{-u}} \left( 2n+1 \right)^n (c+\de)^{n-2} \de^2, \quad \text{ if } j\ne i-1, i,\\[2mm]
\frac{\max_{u\in [0,c+\de]}e^{-u}}{\min_{u\in [0,1]} e^{-u}} \left(2n+1\right)^n (c+\de)^{n-1} 2\de, \quad \text{ if } j\in \{i-1, i\},
\end{cases}
\end{align*}
where the length of $I_i(\yy)=[y_{i-1},y_i]$ is at most $2\de$. This ratio will remain below $C_1\de$, with a sufficiently large constant $C_1$. Therefore, for $t=z_i\in \intt I_i$ we get that 
\[
m_k(\yy) \ge \sum_{j=0}^n |h_j(\yy,s)|+|h_{n+1}(\yy,s)| \ge \frac{1}{C_1\de} \sum_{j=0}^n |h_j(\yy,z_i)|. 
\]
In a similar way, for $t=y_i\in I_i$ we also get, that
\[
m_k(\yy) \ge \sum_{j=0}^n |h_j(\yy,s)|+|h_{n+1}(\yy,s)| \ge \frac{1}{C_1\de} \sum_{j=0}^n |h_j(\yy,y_i)| = \frac{1}{C_1\de},
\]
and hence, 
\begin{align*}
m_k(\yy) & \ge \frac{1}{3C_1\de}\left(2\sum_{j=0}^n |h_j(\yy,z_i)| +1\right)\\
         & \ge \frac{1}{3C_1\de}\left(\sum_{j=0}^n |h_j(\yy,z_i)| +\left|1-\sum_{j=0}^n h_j(\yy,z_i)\right|\right)
          = \frac{1}{3C_1\de} m_i(\yy).
\end{align*}

As before, we can make use of the trivial lower estimate $m_i(\yy)\ge 1$ to infer
\begin{align*}
\max\limits_{l=2,\ldots,n+1} |m_l(\yy)-m_{l-1}(\yy)|  \ge \frac{1}{n} |m_k(\yy)-m_i(\yy)| 
 \ge \frac{1}{n} \frac{|m_k(\yy)-m_i(\yy)|}{m_i(\yy)}
\ge \frac{1}{n} \left( \frac{1}{3C_1\de} -1 \right).
\end{align*}
This means that for small $\de$ we will have $\| \Gamma(\yy)\| \ge \frac{1}{n} \left( \frac{1}{3C_1\de} -1 \right) >C$, whereas $Q$ was bounded by $C$, so that $\Gamma(\yy) \notin Q$. That is, $\xx$ cannot be a limit point of $W$. The proof is finished.
\end{proof}

In the following, we extract and unify the relevant statements from Theorems 1 and 2 of \cite{Shi} that are relevant to our case. 

\begin{theorem}[Shi]\label{th:Shi}
Let $S := \{\xx=(x_1,\dots,x_n) : x_0:=a < x_1 < \dots < x_n < x_{n+1}:=b\}$,
where $-\infty < a < b < \infty$. Let $m_i(\xx) \ge 0$ $(i=1,\dots,n+1)$ be continuously differentiable functions on $S$, and set $\overline{m}(\xx) := \max_i m_i(\xx)$. 

Assume that the functions $m_i$ satisfy the nonsingularity property (Definition \ref{def:nonsing}), and that $\Gamma$, defined in~\eqref{gamma}, satisfies
\begin{equation}\label{lim_gamma}
\lim_{\min\limits_{0\le j\le n-1} (x_{j+1}-x_j)\to 0} \|\Gamma(\xx)\| = \infty.
\end{equation}

Then there exists a unique optimal node system $\yy \in S$, i.e.,
\[
\overline{m}(\yy) = \min_{\xx \in S} \overline{m}(\xx).
\]
Moreover, the Bernstein and Erd\H{o}s Conjectures, as well as the homeomorphism and intertwining properties (see Definitions \ref{def:hom} and \ref{def:int}), hold.

\end{theorem}

We formulate a slight generalization of this result. 

\begin{theorem}\label{Shi_mod}
Let $m_i(\xx) \ge 0$ $(i=1,\dots,n+1)$ be continuously differentiable functions on an open domain $X\subset \RR^n$. With the notations of the above theorem, assume that the functions $m_i$ satisfy the nonsingularity property, and that $\Gamma$ is a proper mapping. 

Then, all the assertions in the conclusion of the above Theorem \ref{th:Shi} hold.
\end{theorem}

\begin{proof}
Theorem \ref{th:Shi} is proved in \cite{Shi} via a linear programming argument, which in itself works also for the case of the simplex belonging to an infinite interval, and as a matter of fact, even for an arbitrary simply connected domain. However, at one point, Hadamard's classical topology theorem is used for $\Gamma$, which in turn requires the nonsingularity and the properness of $\Gamma$. The nonsingularity of $\Gamma$ follows by general linear algebra as given in \cite{Shi}, but actually goes back to \cite{CBoorPinkus}. Further, for a finite interval, properness follows directly from $\Gamma$ satisfying the condition given in \eqref{lim_gamma}. In the case of the simplex built on the infinite interval $[0,\infty)$, or in general for an open domain $X$, properness additionally requires boundedness. (For Lagrange interpolation with  hybrid polynomials, we established it in the first part of the proof of Theorem~\ref{t:homeom_const}). With properness given, Shi's proof can be followed verbatim. 

\end{proof}

We are now prepared to present the desired results on interpolation with the hybrid system. 

\begin{theorem}\label{Bernstein_Erdos}
For the hybrid system $\YY_n=\WW_n\oplus \bf{1}$, the following statements hold true:  
\begin{itemize}

\item[1)] there exists a unique optimal node system $\yy\in S$, i.e., a node system with  
\begin{equation*}
\|L(\yy,\cdot)\|_{\infty}\le \|L(\xx,\cdot)\|_{\infty}, \quad \text{ for any other node system}\quad \xx\in S;
\end{equation*}
\hspace{2mm}
\item[2)] a node system $\yy$ is optimal if and only if it is an equioscillating node system;
\hspace{2mm}
\item[3)] for any $\xx\in S$, $\xx\ne \yy$, there exist indices $i$ and $j$, such that $$m_i(\xx)<m_i(\yy) \quad \text{ and } \quad m_j(\xx)>m_j(\yy);$$ 
\hspace{2mm}
\item[4)] the map $\Gamma\colon S\to \RR^{n}$ defined by \eqref{gamma} is a (global) homeomorphism of $S$ onto $\RR^n$;
\hspace{2mm}
\item[5)] for any two distinct elements $\xx, \zz\in S$, there exist indices $i$ and $j$, such that $$m_i(\xx)<m_i(\zz)\quad \text{ and } \quad m_j(\xx)>m_j(\zz).$$

\end{itemize} 
\end{theorem}

\begin{proof}
We have established that the two crucial properties -- the properness of 
 $\Gamma$ and the nonsingularity of matrices $A_k$, ($k=1,\dots, n+1$)-- hold for the hybrid system $\YY_n$ (see Theorem \ref{t:homeom_const} and Corollary \ref{c:nonsingularity}, respectively). Consequently, the assumptions of Theorem \ref{Shi_mod} are fulfilled and it follows that Parts 1) -- 5) hold.
\end{proof}

\section{Acknowledgement}
The author wishes to express her sincere gratitude to \emph{Professor Szilárd Révész} for introducing her to the present topic and for his continuous support throughout this project. His valuable and thoughtful comments, as well as his recommendations, have greatly contributed to improving this manuscript.

A substantial part of this research was conducted at the \emph{HUN-REN Alfréd Rényi Institute of Mathematics} during the author’s stay as a visiting researcher. The author gratefully acknowledges this support.

This project has received funding from the HUN-REN Hungarian Research Network.


\begin{thebibliography}{AAAA}


\bibitem{Bernstein} 
S. Bernstein,  Sur la limitation des valeurs d'un polyn\^ome $P_n(x)$ de degr\'e $n$ sur tout un segment par ses valeurs en $n+1$ points du segment, {\it Bull. Acad. Sci. URSS Leningrad} (1931), pp. 1025--1050.

\bibitem{Bojanov}
B.~ D.~ Bojanov,
 A generalization of Chebyshev polynomials.
{\it J. Approx. Th.} \textbf{26} (1979), no. 4, 293--300.

\bibitem{Bojanov2002}
B. Bojanov,  Markov-type inequalities for polynomials and splines, in: C. K. Chui, L. L. Schumaker, J. Stökler (Eds.), {\it Approximation Theory X: Abstract and Classical Analysis, Vanderbilt University Press, Nashville, TN,} (2002), 31--90.

\bibitem{CBoorPinkus}
C. de Boor and A. Pinkus,  Proof of the conjectures of Bernstein and Erd\H{o}s concerning the optimal nodes for polynomial interpolation, {\it J. Approx. Th.} \textbf{24} (1978), 289--303.


\bibitem{Erdos-1} 
P. Erd\H os,  Some remarks on polynomials, {\it Bull. Amer. Math. Soc.} {\bf 53} (1947), 1169--l176.

\bibitem{Erdos-2} 
P. Erd\H os,  Problems and results on the theory of interpolation, I, {\it Acta Math. Acad. Sci. Hungar.} {\bf 9} (1958), 381--388.

\bibitem{TLMS2018}
B.~Farkas, B.~Nagy and Sz.~Gy. R\'{e}v\'{e}sz,  A minimax problem for sums of translates on the torus, {\it Trans. London Math. Soc.} \textbf{5} (2018),  no.~1, 1--46.

\bibitem{Homeo} B.~Farkas, B.~Nagy and {\relax Sz}.~{\relax Gy}. R\'ev\'esz,
 A homeomorphism theorem for sums of translates,
{\it Rev. Mat. Complut.} {\bf 37} (2024), no. 2, 341--389.
%
\bibitem{JMAA} B.~Farkas, B.~Nagy and {\relax Sz}.~{\relax Gy}. R\'ev\'esz,
 Fenton type minimax problems for sum of translates functions,
{\it J. Math. Anal. Appl.} {\bf 543} (2025), no. 2, Paper No. 128931, 25 pp.


\bibitem{Haar}
A. Haar,  Die Minkowskische Geometrie und die Annäherung an stetige Funktionen, {\it Math. Ann.} \textbf{78} (1918), 249--311.


\bibitem{KarlinStudden}
S. Karlin and W. J. Studden, {\it Tchebycheff Systems: with Applications in Analysis and Statistics}, Interscience, New York, (1966).

\bibitem{KilgorePhD}
T. Kilgore,  Optimization of the Lagrange interpolation process, (Ph.D. Thesis, The University of Texas at Austin), {\it ProQuest LLC, Ann Arbor, MI}, (1974), 52 pp. {\tt https://www.proquest.com/docview/302705808?sourcetype=Dissertations Theses}


\bibitem{Kilgore-BullAMS}
T. Kilgore,  Optimization of the norm of the Lagrange interpolation operator, {\it Bull. Amer. Math. Soc.} {\bf 83}, 5, September (1977), 1069--1071.

\bibitem{Kilgore}
T. Kilgore,  A characterization of the Lagrange interpolating projection with minimal Tchebycheff norm, {\it J. Approx. Th.} \textbf{24} (1978), 273--288.

\bibitem{Kilgore1985}
T. Kilgore,  A Note on Functions with Interlacing Roots, {\it J. Approx. Th.} \textbf{43} (1985), 25--28.

\bibitem{K-AMH}
T. Kilgore,  Optimal interpolation with exponentially weighted polynomials on an unbounded interval, {\it Acta Math. Hung.} \textbf{57}, (1991), 85--90.

\bibitem{Kilgore-Cheney}
T. Kilgore and E. W. Cheney,  A theorem on interpolation in Haar subspaces, {\it Aequationes Math.} \textbf{14} (1976), 391--400.


\bibitem{Markov}
V. A. Markov,  On the Functions of Least Deviation from Zero in a Given Interval, Izdat. Imp. Akad. Nauk, St. Petersburg, 1892 (in Russian); German translation with condensation: W. Markoff, Über polynome die einen gegebenen Intervalle möglichst wenig von Null abweichen, {\it Math. Ann.} \textbf{77} (1916), 213--258.

\bibitem{MN}
L. Milev and N. Naidenov,  Markov interlacing property for exponential polynomials,
{\it J. Math. Anal. Appl.} {\bf 367} (2010), 669--676.

\bibitem{Morris-Cheney} P. D. Morris and E. W. Cheney,  On the existence and characterization of minimal projections, {\it J. Reine Angew. Math.} {\bf 270}  (1974), 61--76.

\bibitem{Tatiana-II} T. M. Nikiforova,  Homeomorphism theorem for sums of translates on the real axis, Preprint, see at {\tt https://arxiv.org/abs/2509.07776}.

\bibitem{Tatiana-I} T. M. Nikiforova,  Minimax and maximin problems for sums of translates on the real axis, {\it J. Approx. Th.} {\bf 311} (2025), Paper No. 106190, 19 pp.

\bibitem{Shi}
Y. G. Shi,  A minimax problem admitting the equioscillation characterization of Bernstein and Erd\H os, {\it J. Approx. Th.} \textbf{92} (1998), 463--471.

\bibitem{SzV} J. Szabados and P. Vértesi, {\it Interpolation of functions}, World Scientific Publishing, Teaneck, NJ, (1990).


\end{thebibliography}
\end{document}